\documentclass[11pt]{article}
\usepackage{amsfonts}
\usepackage{mathrsfs}
\usepackage{amsthm}
\usepackage{amsmath}
\usepackage{amssymb}
\usepackage{latexsym}
\usepackage{graphicx}
\usepackage{color,varioref}
\usepackage{longtable}
\usepackage{cite}
\usepackage{enumerate}
\definecolor{red}{rgb}{0.9,0,0}

\definecolor{green}{rgb}{0,0.9,0}

\definecolor{blue}{rgb}{0,0,0.9}

 \def\R{{\cal R}}

\def\tilde{\widetilde}
\def\Limsup{\mathop{{\rm Lim}\,{\rm sup}}}

\newtheorem{theorem}{Theorem}[section]
\newtheorem{proposition}{Proposition}[section]
\newtheorem{lemma}{Lemma}[section]
\newtheorem{corollary}{Corollary}[section]

\newtheorem{definition}{Definition}[section]

\newtheorem{example}{Example}[section]
\def\sce{\setcounter{equation}{0}}
\numberwithin{equation}{section}

\begin{document}

\title{\bf Complete Characterizations of Well-Posedness in Parametric Composite Optimization}

\author{Boris S. Mordukhovich\footnote{Department of Mathematics and Institute for Artificial Intelligence and Data Science, Wayne State University, Detroit, MI 48202 (aa1086@wayne.edu).} \quad Peipei Tang\footnote{Corresponding author. School of Computer and Computing Science, Hangzhou City University, Zhejiang Key Laboratory of Big Data Intelligent Computing, Hangzhou, 310015, China (tangpp@hzcu.edu.cn).}\quad Chengjing Wang\footnote{School of Mathematics, Southwest Jiaotong University, Chengdu, China (renascencewang@hotmail.com).}}

\maketitle

\vspace*{-0.3in}
\begin{center}
{\bf Dedicated to the memory of Franco Giannessi,\\ an outstanding mathematician and human being}
\end{center}

\begin{abstract}
\noindent This paper characterizes the well-posedness of Karush-Kuhn-Tucker system for perturbed composite optimization. Using the parabolic regularity, we introduce a novel second-order variational function, shown to be the pivotal object governing second-order behavior. This foundational result yields the strong second-order sufficient condition introduced here for the general class of composite optimization problems to naturally extend the classical second-order sufficient condition in nonlinear programming. Then we obtain several equivalent characterizations of the second-order qualification condition (SOQC) and  highlight its relationship with the constraint nondegeneracy condition under the $\mathcal{C}^{2}$-cone reducibility assumption. These insights lead us to multiple equivalent conditions for the major Lipschitz-like/Aubin property of KKT systems, including SOQC combined with the new second-order subdifferential condition and SOQC combined with tilt stability of local minimizers. Under the $\mathcal{C}^{2}$-cone reducibility, we settle the long-standing question by proving the equivalence between the Aubin property of KKT systems around local minimizers and the classical notion  of strong regularity. Finally, we demonstrate that the Lipschitz-like property is equivalent to the nonsingularity of the generalized Jacobian associated with the KKT system under a certain verifiable assumption. These results provide a unified and rigorous framework for analyzing stability and sensitivity of solutions to composite optimization problems, as well as for the design and justification of numerical algorithms.

\end{abstract}\vspace*{-0.1in}
\vspace*{0.1in}
\vspace*{-0.2in}

\section{Introduction}\label{intro}\sce\vspace*{-0.05in}

This paper addresses the areas of variational analysis, generalized differentiation, and stability in optimization in which contributions by Professor Giannessi are difficult to overstate. We particularly concern the study of the general class of {\em composite optimization} problems given by 
\begin{align}\label{comp-prob}
\qquad\min_{x\in\mathcal{R}^{n}} \varphi(x):=h(x)+g(F(x)), 
\end{align}
where $h:\mathcal{R}^{n}\rightarrow\mathcal{R}$, $F:\mathcal{R}^{n}\rightarrow\mathcal{R}^{m}$ are ${\cal C}^{2}$-smooth mappings and $g:\mathcal{R}^{m}\rightarrow \overline{\mathcal{R}}:=\mathcal{R}\cup\{\infty\}$ is an extended-real-valued lower semicontinuous (l.s.c.) convex function, which is always assumed to be proper, i.e., $g(x)>-\infty$ for all $x\in\mathcal{R}^{m}$ with ${\rm dom}\,g:=\{x\in\mathcal{R}^{m}\;|\;g(x)<\infty\}\ne\emptyset$. The extended-real-valued format in \eqref{comp-prob} allows us to incorporate various constraints via $F(x)\in{\rm dom}\,g$. In particular, model \eqref{comp-prob} covers problems of nonlinear programming, second-order cone programming, and semidefinite programming that correspond to the indicator function $g$ of the corresponding cone, as well as other classes of constrained problems considered below.

We study the {\em perturbed} version of \eqref{comp-prob}
with the parameters $a\in\mathcal{R}^n$ and $b\in\mathcal{R}^m$ written as
\begin{align}\label{comp-prob-perturb}
\qquad\min_{x\in\mathcal{R}^{n}} \tilde{\varphi}(x,b)-\langle a,x\rangle, 
\end{align}
where $\tilde{\varphi}(x,b):=h(x)+g(F(x)+b)$. The {\em KKT system} (necessary optimality conditions) associated with the canonically perturbed problem \eqref{comp-prob-perturb} takes the form 
\begin{align}\label{eq:kkt-perturbed-cop}
\nabla_{x}L(x,u)=a,\quad b\in\partial g^{*}(u)-F(x),
\end{align}
where $g^*$ is the {\em Fenchel conjugate} of $g$, and where $L(\cdot,\cdot)$ is the {\em Lagrangian} of \eqref{comp-prob} defined by 
\begin{align*}
L(x,u):=h(x)+\langle u,F(x)\rangle,\quad (x,u)\in\mathcal{R}^{n}\times\mathcal{R}^{m}.
\end{align*}
Denote by $S_{\operatorname{KKT}}(a,b)$ the parameterized {\em solution mapping} for \eqref{eq:kkt-perturbed-cop} and by $\Phi$ the {\em KKT mapping}
\begin{equation}\label{Phi}
\Phi(x,u):=\left[\begin{array}{c}
\nabla_{x}L(x,u)\\
\partial g^{*}(u)-F(x)
\end{array}\right].
\end{equation}

During the past few decades, great many efforts have been made to investigate diverse {\em well-posedness} issues for constraint and variational systems under perturbations; see, e.g., the monographs \cite{Bonnans2000,DontchevandRockafellar2014,KlatteKummer2002,Mordukhovich2006,Mordukhovich2018,Mordukhovich2024,Rockafellar1998} and the references therein. In his landmark work \cite{Robinson1980}, Robinson introduced the notion of {\em strong regularity} for the generalized equation $0\in\Phi(x,u)$ at $(\bar{x},\bar{u})$, which means that the inverse of the linearized mapping
\begin{equation}\label{linearized-Phi}
(x,u)\mapsto\left[\begin{array}{c}
\nabla_{xx}^{2}L(\bar{x},\bar{u})(x-\bar{x})+\nabla F(\bar{x})^{T}(u-\bar{u})\\
-F(\bar{x})-\nabla F(\bar{x})(x-\bar{x})
\end{array}\right]+\left[\begin{array}{c}
0\\
\partial g^{*}(u)
\end{array}\right]
\end{equation} 
admits a Lipschitz continuous single-valued localization around the origin, where ``$A^T$" stands for the matrix transposition. Characterizing strong regularity of perturbed KKT systems at their solution points has emerged as a significant research topic in mathematical programming. In the context of nonlinear programming (NLP), it is established in \cite[Theorem~4.1]{Robinson1980} and \cite[Theorem 4.10]{BonnansSulem1995} that strong regularity of the canonically perturbed KKT system is equivalent to the simultaneous fulfillment of the strong second-order sufficient condition (SSOSC) and the linear independence constraint qualification (LICQ), as well as \cite[Theorem~3.1]{Robinson1980} to the nonsingularity of  the associated Clarke's generalized Jacobian of the KKT system at the corresponding solution point.  

As shown in  \cite[Theorem~5]{DontchevRockafellar1996}, strong regularity of  KKT systems associated with NLPs is equivalent to the Lipschitzian behavior of mappings that was first revealed by Aubin \cite{Aubin1984} in the context of convex programming under the name of ``pseudo-Lipschitz property". This notion, which is a graphical localization of the classical (Hausdorff) Lipschitz continuity of set-valued mappings, is now called the {\em Aubin} or {\em Lipschitz-like property} to emphasize its Lipschitzian nature (while ``pseudo" means ``false"). It has been well realized that such Lipschitzian behavior of set-valued mappings/multifunctions is equivalent to the two major well-posedness concepts in nonlinear and variational analysis known as {\em metric regularity} and {\em linear openness/covering} applied to the mapping inverse; see, e.g., \cite{Mordukhovich2006,Rockafellar1998} for more details and references. 

Developing the line of \cite{Robinson1980}, the equivalence between strong regularity and appropriate versions of SSOSC and nondegeneracy conditions is established in \cite[Theorem~30]{BonnansRamirezC2005} for second-order cone programming, while the equivalence between strong regularity and Aubin property for such programs is obtained in \cite[Theorem~21]{OutrataJivriv2011} (see also the Errarum in \cite{Opazo}) under the strict complementarity condition and  in  \cite[Theorem~4.12]{ChenChenSunZhu2025} without the latter. For semidefinite programming, similar results are available in \cite[Theorem~4.1]{Sun2006} and \cite[Theorem~4.1]{ChenChenSunZhang2025}. For general constrained optimization problems, the results of \cite[Theorem~5.6]{MordukhovichNghiaRockafellar2015} (see also \cite[Section~5.5]{Mordukhovich2024}) establish equivalence among the properties of strong regularity, Lipschitzian strong stability, Lipschitzian full stability, and second-order subdifferential conditions at local minimizers under the reducibility and nondegeneracy conditions. Subsequently, \cite[Theorem~3.3]{Mordukhovich2017} derives the equivalence between the notions of Lipschitzian full stability for local minimizers under some nondegeneracy and of strong regularity of the associated KKT system for a class of composite optimization problems involving convex piecewise linear functions. The recent paper \cite{Tangwang2024} proposes a definition of the generalized SSOSC and demonstrates the equivalence between the introduced SSOSC in conjunction with the nondegeneracy condition, the nonsingularity of Clarke’s generalized Jacobian of the nonsmooth system at a  given KKT point, and strong regularity of the KKT point for composite optimization problems. This comprehensive body of research demonstrates fundamental connections between various regularity and stability concepts in optimization theory, providing a unified framework for analyzing well-posedness behavior of optimization problems and their solution structures.

Recall \cite{Bonnans2000} that the nonnegative orthant cone, the second-order cone and the positive semidefinite cone are $\mathcal{C}^{2}$-cone reducible. This naturally raises the question of whether the Lipschitz-like property of the KKT system is equivalent to strong regularity of solution points for all $\mathcal{C}^{2}$-cone reducible problems and also in more general settings. While significant steps in this direction have been recently made in \cite{ChenChenSunZhu2025,ChenChenSunZhang2025} for second-order cone and semidefinite programming based on the Mordukhovich criterion, it is mentioned in the conclusion of \cite{ChenChenSunZhang2025} that  ``currently, it is not clear to us how to extend these results to the generic non-polyhedral ${\cal C}^2$-cone reducible constrained optimization." Nevertheless, the latest preprint \cite{sun25} provides the equivalence between the Aubin property and strong regularity for certain ${\cal C}^2$-cone reducible generalized equations by using the degree theory, without any appeal to variational analysis.

In this paper, we employ and develop advanced tools of variational analysis and generalized differentiation that allow us to obtain comprehensive characterizations of well-posedness properties of the KKT systems \eqref{eq:kkt-perturbed-cop} with ${\cal C}^2$-cone reducible functions $g$. In particular, we establish in this way the equivalence between strong regularity and Aubin property for the KKT system that is applicable to our model when $g$ is the indicator function of a ${\cal C}^2$-cone reducible set.

As has been well realized, primal and dual second-order conditions play a pivotal role in characterizing well-posedness of optimization problems. In what follows, we systematically exploit the powerful property of {\em parabolic regularity} for the function $g$ in \eqref{comp-prob-perturb} that has been largely investigated in the recent publications \cite{MohammadiMordukhovichSarabi2020,Mohammadi2021,Mohammadi2022,Mohammadi2020}. To derive our main results, we provide a precise description of the second-order qualification condition (SOQC) and establish its relationship with the nondegeneracy condition under the $\mathcal{C}^{2}$-cone reducibility. Leveraging the properties of isolated calmness and tilt stability, we characterize the Aubin property based on two distinct conditions. Under the parabolic regularity of $g$, we show that this property is equivalent to SOQC combined with the new second-order subdifferential condition. Imposing the $\mathcal{C}^{2}$-cone reducibility condition, we prove that the Aubin property of the reference KKT system is equivalent to its strong regularity. Moreover, under the parabolic regularity of $g$ and a certain mild assumption, which holds for many common functions such as the indicator function of the positive semidefinite cone, the nuclear norm function, and the spectral norm functions, we establish the equivalence between the Aubin property and the nonsingularity of the generalized Jacobian associated with the KKT system. These results provide novel and multifaceted perspectives on characterizing Lipschitzian behavior and related well-posedness properties, while opening new avenues for understanding their connections to various stability and regularity conditions in optimization theory and algorithms.

The remainder of this paper is organized as follows. Section~\ref{sec:preliminary} introduces and discusses fundamental concepts of variational analysis and generalized differentiation, which serve as basic tools for the subsequent developments. In Section~\ref{sec:srcq-nondegeneracy-sosc-ssosc}, we define  the constraint nondegeneracy condition as well as the general SOSC version for composite optimization problems. We also formulate the general version of SSOSC for composite models and establish its relationship with SOSC under the parabolic regularity. Section~\ref{sec:characterization-soqc} provides detailed characterizations of the SOQC property under the assumptions of parabolic regularity and $\mathcal{C}^{2}$-cone reducibility. Section~\ref{sec:equivalence-aubin-property} is devoted to deriving comprehensive equivalent characterizations of the well-posedness properties for the KKT system \eqref{eq:kkt-perturbed-cop}. Specifically, we establish equivalence of the Aubin property to (i) SOQC combined with the second-order subdifferential condition, to (ii) SOQC combined with the tilt stability, to (iii) strong regularity, and to (iv) the nonsingularity of the KKT system under fairly mild assumptions. The concluding Section~\ref{sec:Conclusion} summarizes the main achievements of the paper and discusses some directions of our future research  and applications.

Our notation is standard in variational analysis; see \cite{Mordukhovich2006,Rockafellar1998}. The closed unit ball in $\mathcal{R}^{m}$ is denoted by $\mathbb{B}_{\mathcal{R}^{m}}$. For a given point $x\in\mathcal{R}^{m}$ and a number $\varepsilon>0$, the closed ball centered at $x$ with radius $\varepsilon$ is defined by $\mathbb{B}(x,\varepsilon):=\{u\in\mathcal{R}^{m}\, |\, \|u-x\|\leq\varepsilon\}$. Given a set $C\subset\mathcal{R}^{m}$,  the indicator function $\delta_{C}$ of $C$ is defined by $\delta_{C}(x):=0$ if $x\in C$ and $\delta_{C}(x):=\infty$ otherwise. For a point $x\in\mathcal{R}^{m}$, $\operatorname{dist}(x,C)$ denotes the distance from $x$ to the set $C$, while $\operatorname{\Pi}_{C}(x)$ stands for the projection of $x$ onto $C$. For a closed convex cone $C\ne\emptyset$, the lineality space of $C$ is $\operatorname{lin}C:=C\cap(-C)$ representing the largest subspace contained in $C$, and the affine hull $\operatorname{aff}C:=C-C$ is the smallest subspace containing $C$. Given a matrix $A\in\mathcal{R}^{m\times n}$, the kernel of $A$ is $\operatorname{ker}A:=\{x\in\mathcal{R}^{n}\,|\,Ax=0\}$, the range of $A$ is $\operatorname{rge}A$, and $A^{\dagger}$ denotes the Moore-Penrose inverse. The notation $C_1\subset C_2$ means that $C_1$ belongs or equals to $C_2$. \vspace*{-0.15in}

\section{Basic Tools of Variational Analysis}\label{sec:preliminary}\vspace*{-0.05in}

In this section, we recall some fundamental tools of variational analysis and generalized differentiation while referring  to the monographs \cite{Bonnans2000,Mordukhovich2006,Rockafellar1998} for more details.

Consider $f:\mathcal{R}^{m}\rightarrow\overline{\mathcal{R}}$ identified with the epigraph $\operatorname{epi}f:=\{(x,\alpha)\in\mathcal{R}^{m}\times\mathcal{R}\,|\,f(x)\leq\alpha\}$. 
The (Fr\'{e}chet) {\em regular subdifferential} of $f$ at $\bar{x}\in{\rm dom}\,f$ is defined by
\begin{equation}\label{rsub}
\widehat{\partial} f(\bar{x}):=\left\{v\in\mathcal{R}^{m}\,\left|\,\liminf_{x\rightarrow\bar{x}}\frac{f(x)-f(\bar{x})-\langle v,x-\bar{x}\rangle}{\|x-\bar{x}\|}\geq0\right.\right\},
\end{equation}
and the (Mordukhovich) {\em limiting subdifferential} of $f$ at $\bar{x}$ is
\begin{equation}\label{lsub}
\partial f(\bar{x}):=\big\{v\in\mathcal{R}^{m}\;\big|\,\exists\,x_k\to\bar x,\;v_k\to v\;\mbox{ with }\;f(x^{k})\rightarrow f(\bar{x}),\;v^{k}\in\widehat{\partial}f(x^{k})\big\}.
\end{equation}
Given an l.s.c.\ function $f$ and a number $\sigma>0$, the {\em proximal mapping} $\operatorname{Prox}_{\sigma f}$ corresponding to $f$ with parameter $\sigma$ is defined by
\begin{align}\label{prox}
\operatorname{Prox}_{\sigma f}(x)&:=\mathop{\arg\min}_{u\in\mathcal{R}^{m}}\Big\{f(u)+\frac{1}{2\sigma}\|u-x\|^{2}\Big\},\quad x\in\mathcal{R}^{m}.
\end{align}
When $f$ is convex, it follows from \cite[Proposition~13.37]{Rockafellar1998} that  \eqref{prox} is single-valued and Lipschitz continuous on $\mathcal{R}^{m}$. The (first-order) {\em subderivative} of $f:\mathcal{R}^{m}\rightarrow\overline{\mathcal{R}}$ at $\bar{x}\in{\rm dom}\,f$ in the direction $d\in\mathcal{R}^{m}$ is defined by
\begin{align*}
df(\bar{x})(d):=\liminf_{{\tau\downarrow0}\atop{d'\rightarrow d}}\frac{f(\bar{x}+\tau d')-f(\bar{x})}{\tau}.
\end{align*}
When $df(\bar{x})(d)$ is finite, the {\em parabolic subderivative} of $f$ at $\bar{x}$ relative to $d$ and $w\in\mathcal{R}^{m}$ is 
\begin{align}\label{parab-subder}
d^{2}f(\bar{x})(d,w):=\liminf_{{\tau\downarrow0}\atop{w'\rightarrow w}}\frac{f(\bar{x}+\tau d+\frac{1}{2}\tau^{2}w')-f(\bar{x})-\tau df(\bar{x})(d)}{\frac{1}{2}\tau^{2}}.
\end{align}
We say that $f$ is {\em parabolically epi-differentiable} at $\bar{x}$ for $d$ if $\operatorname{dom}d^{2}df(\bar{x})(d,\cdot)\neq\emptyset$ and for every $w\in\mathcal{R}^{m}$ and $\tau_{k}\downarrow 0$, there exists $w^{k}\rightarrow w$ such that
\begin{align*}
d^{2}f(\bar{x})(d,w)=\lim_{k\rightarrow\infty}\frac{f(\bar{x}+\tau_{k} d+\frac{1}{2}\tau_{k}^{2}w^{k})-f(\bar{x})-\tau_{k} df(\bar{x})(d)}{\frac{1}{2}\tau_{k}^{2}}.
\end{align*}
Associate with $f$ the parametric family of {\em second-order difference quotients} at $\bar{x}$ for $\bar{u}$ by
\begin{align*}
\Delta_{\tau}^{2}f(\bar{x},\bar{u})(d):=\frac{f(\bar{x}+\tau d)-f(\bar{x})-\tau\langle\bar{u},d\rangle}{\frac{1}{2}\tau^{2}}\quad\mbox{with}\ d\in\mathcal{R}^{m},\ \tau>0.
\end{align*}
The {\em second-order subderivative} of $f$ at $\bar{x}\in\operatorname{dom}f$ for $\bar{u},d\in\mathcal{R}^{m}$ is defined by
\begin{align}\label{2subder}
d^{2}f(\bar{x},\bar{u})(d):=\liminf_{{\tau\downarrow0}\atop{d'\rightarrow d}}\Delta_{\tau}^{2}f(\bar{x},\bar{u})(d').
\end{align}
The function $f$ is said to be {\em twice epi-differentiable} at $\bar{x}$ for $\bar{u}$ if for every sequence $\tau_{k}\downarrow0$ and every $d\in\mathcal{R}^{m}$, there exists $d^{k}\rightarrow d$ such that
\begin{align*}
d^{2}f(\bar{x},\bar{u})(d)=\lim_{k\rightarrow\infty}\Delta_{\tau_{k}}^{2}f(\bar{x},\bar{u})(d^{k}).
\end{align*}
The next property plays a crucial role in deriving the main results of this paper. \vspace*{-0.05in}

\begin{definition}\label{def:par-reg}
 It is said that $f:\mathcal{R}^{m}\rightarrow\overline{\mathcal{R}}$ is {\sc parabolically regular} at $\bar{x}\in\operatorname{dom}f$ for $\bar{u}\in\mathcal{R}^m$ in direction $d\in\mathcal{R}^m$ 
 if we have the second-order relationship
\begin{align*}
\inf_{w\in\mathcal{R}^{m}}\left\{d^{2}f(\bar{x})(d,w)-\langle\bar{u},w\rangle\right\}=d^{2}f(\bar{x},\bar{u})(d).
\end{align*}
\end{definition}

It has been realized that parabolic regularity holds for broad classes of extended-real-valued functions; see \cite{Bonnans2000,MohammadiMordukhovichSarabi2020,Mohammadi2021,Mohammadi2022,Mohammadi2020,Rockafellar1998}. The following proposition demonstrates that parabolic regularity of a convex function is inherited by the indicator function of its epigraph.\vspace*{-0.05in}

\begin{proposition}\label{prop-parabolically-regular-delta-epi-f}
Let $f:\mathbb{R}^{m}\rightarrow\overline{\mathcal{R}}$ be an l.s.c.\ convex function with $\bar{x}\in\operatorname{dom}f$, let $\bar{\mu}\in\partial f(\bar{x})$, and let $\lambda>0$, Suppose that $f$ is parabolically regular at $\bar{x}$ for $\bar{\mu}$ and parabolically epi-differentiable at $\bar{x}$ in every direction $d\in\{d\,|\,df(\bar{x})(d)=\langle\bar{\mu},d\rangle\}$ for $\bar{\mu}$. Then we have the equality
\begin{align*}
\operatorname{dom}d^{2}\delta_{\operatorname{epi}f}((\bar{x},f(\bar{x})),\lambda(\bar{\mu},-1))=\left\{(d,\gamma)\,\left|\,df(\bar{x})(d)=\langle\bar{\mu},d\rangle=\gamma\right.\right\}
\end{align*}
and deduce therefore that the indicator function $\delta_{\operatorname{epi}f}$ is parabolically regular at $(\bar{x},f(\bar{x}))$ in every direction $(d,\gamma)\in\left\{(d,\gamma)\,\left|\,df(\bar{x})(d)=\langle\bar{\mu},d\rangle=\gamma\right.\right\}$ for $\lambda(\bar{\mu},-1)$.
\end{proposition}
\begin{proof}
It follows from \cite[Theorem 8.9]{Rockafellar1998} that $\lambda(\bar{\mu},-1)\in\partial \delta_{\operatorname{epi}f}(\bar{x},f(\bar{x}))$. Combining \cite[Example~13.62]{Rockafellar1998}, \cite[Theorem~3.3]{Mohammadi2021}, \cite[Theorem~8.2]{Rockafellar1998}, and \cite[Proposition~2.126]{Bonnans2000} tells us that
\begin{align*}
\operatorname{dom}d^{2}\delta_{\operatorname{epi}f}((\bar{x},f(\bar{x})),\lambda(\bar{\mu},-1))&=\left\{(d,\gamma)\,\left|\,d\delta_{\operatorname{epi}f}(\bar{x},f(\bar{x}))(d,\gamma)=\lambda\langle\bar{\mu},d\rangle-\lambda\gamma\right.\right\}\\
&=\left\{(d,\gamma)\,\left|\,df(\bar{x})(d)=\langle\bar{\mu},d\rangle=\gamma\right.\right\}.
\end{align*}
By the definition of $d^{2}\delta_{\operatorname{epi}f}((\bar{x},f(\bar{x})),\lambda(\bar{\mu},-1))$, for any $(d,\gamma)\in\operatorname{dom}d^{2}\delta_{\operatorname{epi}f}((\bar{x},f(\bar{x})),\lambda(\bar{\mu},-1))$, there exist sequences $\tau_{k}\downarrow0$ and $(d^{k},\gamma^{k})\rightarrow(d,\gamma)$ with $(\bar{x}+\tau_{k}d^{k},f(\bar{x})+\tau_{k}\gamma^{k})\in\operatorname{epi}f$ such that
\begin{align*}
&d^{2}\delta_{\operatorname{epi}f}((\bar{x},f(\bar{x})),\lambda(\bar{\mu},-1))(d,\gamma)=\lim_{k\rightarrow\infty}-\frac{\tau_{k}\lambda\langle\bar{\mu},d^{k}\rangle-\tau_{k}\lambda\gamma^{k}}{\frac{1}{2}\tau_{k}^{2}}\\
&\geq\lim_{k\rightarrow\infty}\lambda\frac{f(\bar{x}+\tau_{k}d^{k})-f(\bar{x})-\tau_{k}\langle\bar{\mu},d^{k}\rangle}{\frac{1}{2}\tau_{k}^{2}}=\lambda d^{2}f(\bar{x},\bar{\mu})(d),
\end{align*}
where the last equality is derived from the twice epi-differentiability of $f$ based on \cite[Proposition~3.2]{MordukhovichTangWang-tilt-2025}. Moreover, it follows from \cite[Proposition~3.2]{Mohammadi2021} and \cite[Proposition~3.41]{Bonnans2000} that
\begin{align*}
d^{2}\delta_{\operatorname{epi}f}((\bar{x},f(\bar{x})),\lambda(\bar{\mu},-1))(d,\gamma)&\leq\inf_{(w,\beta)\in\mathcal{R}^{m}\times\mathcal{R}}\left\{d^{2}\delta_{\operatorname{epi}f}(\bar{x},f(\bar{x}))((d,\gamma),(w,\beta))-\lambda\langle\bar{\mu},w\rangle+\lambda\beta\right\}\\&=\inf_{(w,\beta)\in\mathcal{T}^{2}_{\operatorname{epi}f}((\bar{x},f(\bar{x})),(d,\gamma))}\{\lambda\langle\bar{\mu},w\rangle-\lambda\beta\}=\lambda d^{2}f(\bar{x},\bar{\mu})(d).
\end{align*}
Therefore,  we arrive at the representation
\begin{align*}
d^{2}\delta_{\operatorname{epi}f}((\bar{x},f(\bar{x})),\lambda(\bar{\mu},-1))(d,\gamma)=\inf_{(w,\beta)\in\mathcal{R}^{m}\times\mathcal{R}}\left\{d^{2}\delta_{\operatorname{epi}f}(\bar{x},f(\bar{x}))((d,\gamma),(w,\beta))-\lambda\langle\bar{\mu},w\rangle+\lambda\beta\right\},
\end{align*}
verifying the parabolic regularity of
$\delta_{\operatorname{epi}f}$ at $(\bar{x},f(\bar{x}))$ for $\lambda(\bar{\mu},-1)$ and completing the proof.
\end{proof}

Now we pass to sets and define for them tangent and normal approximations associated with subderivative and subdifferentials of extended-real-valued functions. Given a set $C\subset\mathcal{R}^{m}$ with $\bar{x}\in C$, the (Bouligand-Severi) {\em tangent/contingent cone} to $C$ at $\bar{x}$ is 
\begin{align*}
\mathcal{T}_{C}(\bar{x}):=\big\{d\in\mathcal{R}^{m}\,\big|\,\exists\,\tau_{k}\downarrow0,\ d^{k}\rightarrow d\ \mbox{as}\ k\rightarrow\infty\ \mbox{with}\ \bar{x}+\tau_{k}d^{k}\in C\big\}.
\end{align*}
The (Fr\'{e}chet) {\em regular normal cone} $\widehat{\mathcal{N}}_{C}(\bar{x})$ and the (Mordukhovich) {\em limiting normal cone} $\mathcal{N}_{C}(\bar{x})$ to $C$ at $\bar{x}$ are defined, respectively, by 
\begin{equation}\label{nc}
\begin{array}{ll}
\widehat{\mathcal{N}}_{C}(\bar{x}):=\big\{d\,\big|\,\langle d,x-\bar{x}\rangle\leq o(\|x-\bar{x}\|)\ \mbox{for}\ x\in C\big\},\\
\mathcal{N}_{C}(\bar{x}):=\big\{
d\,\big|\,\exists\,x^{k}\rightarrow\bar{x}\ \mbox{with}\ x^{k}\in C\ \mbox{and}\ d^{k}\rightarrow d\ \mbox{with}\ d^{k}\in\widehat{\mathcal{N}}_{C}(x^{k})\big\}.
\end{array}
\end{equation}
When $C$ is a closed convex cone and $\bar{x}\in C$, it admits that 
\begin{equation}\label{normal-cone-closed-convex-cone}
    \mathcal{N}_{C}(\bar{x})=\{\bar{x}\}^{\perp}\cap C^{\circ}.
\end{equation}
Consider next a set-valued mapping $S:\mathcal{R}^{m}\rightrightarrows\mathcal{R}^{n}$ with the domain and graph defined by 
\begin{align*}
\operatorname{dom}S:=\left\{x\in\mathcal{R}^{m}\,\left|\,S(x)\neq\emptyset\right.\right\}\;\mbox{ and }\;\operatorname{gph}S:=\left\{(x,y)\in\mathcal{R}^{m}\times\mathcal{R}^{n}\,\left|\,y\in S(x)\right.\right\},
\end{align*}
respectively. The {\em graphical derivative} of $S$ at $\bar{x}$ for $\bar{u}\in S(\bar{x})$ is given by
\begin{align}\label{gr-der}
DS(\bar{x},\bar{u})(w):=\left\{z\in\mathcal{R}^{n}\,\left|\,(w,z)\in\mathcal{T}_{\operatorname{gph}S}(\bar{x},\bar{u})\right.\right\},\quad w\in\mathcal{R}^{m}.
\end{align}
The {\em strict graphical derivative} of $S$ at $\bar{x}$ for $\bar{u}\in S(\bar{x})$ is
\begin{align}\label{sgr-der}
D_{*}S(\bar{x},\bar{u})(w):=\Big\{z\in\mathcal{R}^{n}\,&\Big|\,\exists\,\tau_{k}\downarrow0,\ (x^{k},u^{k})\rightarrow(\bar{x},\bar{u})\ \mbox{with}\ u^{k}\in S(x^{k}),\\
&w^{k}\rightarrow w\ \mbox{with}\ z^{k}\in\frac{S(x^{k}+\tau_{k}w^{k})-u^{k}}{\tau_{k}},\ z^{k}\rightarrow z\Big\}.\nonumber
\end{align}
The {\em regular coderivative} and {\em limiting coderivative} of $S$ at $\bar{x}$ for $\bar{u}\in S(\bar{x})$ are defined via the corresponding normal cone in \eqref{nc} by
\begin{equation}\label{cod}
\begin{array}{ll}
\widehat{D}^{*}S(\bar{x},\bar{u})(y):=\big\{v\in\mathcal{R}^{m}\,\big|\,(v,-y)\in\widehat{\mathcal{N}}_{\operatorname{gph}S}(\bar{x},\bar{u})\big\},\quad y\in\mathcal{R}^{n},\\
D^{*}S(\bar{x},\bar{u})(y):=\left\{v\in\mathcal{R}^{m}\,\left|\,(v,-y)\in\mathcal{N}_{\operatorname{gph}S}(\bar{x},\bar{u})\right.\right\},\quad y\in\mathcal{R}^{n}.
\end{array}
\end{equation}
If $S$ is single-valued, we omit the second argument in the notations above. 

For an open set $\mathcal{U}\subset\mathcal{R}^{m}$, denote  by $\mathcal{U}_{F}$ the subset of $\mathcal{U}$ where the given single-valued mapping $F:\mathcal{U}\rightarrow\mathcal{R}^{n}$ is Fr\'{e}chet differentiable. 
If $F$ is locally Lipschitzian around $\bar x\in\mathcal{U}$, the {\em $B$-subdifferential} of $F$ at $\bar{x}$ is defined by
\begin{align}\label{B-sub}
\partial_{B}F(\bar{x}):=\left\{A\in\mathcal{R}^{n\times m}\,\left|\,\exists\,x^{k}\rightarrow\bar{x}\ \mbox{with}\ x^{k}\in\mathcal{U}_{F},\ \nabla F(x^{k})\rightarrow A\right.\right\},
\end{align}
where $\nabla F$ stands for the classical Jacobian. The (Clarke) {\em generalized Jacobian} $\mathcal{J}F(\bar{x})$ of $F$ at $\bar{x}$ is the convex hull of $\partial_{B}F(\bar{x})$. The following are simple observations for the inverses
\begin{align}\label{inverse-graphical-derivative}
w\in DS^{-1}(\bar{u},\bar{x})(z)&\Longleftrightarrow z\in DS(\bar{x},\bar{u})(w),\\
y\in D^{*}S^{-1}(\bar{u},\bar{x})(v)&\Longleftrightarrow -v\in D^{*}S(\bar{x},\bar{u})(-y),\label{inverse-coderivative}
\end{align}
and summation rules for the graphical derivatives \eqref{gr-der}, \eqref{sgr-der} and the limiting coderivative in \eqref{cod} of $S:\mathcal{R}^{m}\rightrightarrows\mathcal{R}^{n}$ with the ${\cal C}^1$-smooth addition $h\colon\mathcal{R}^m\to\mathcal{R}^n$:
\begin{align}\label{sum-rule-graphical-derivative}
D(h+S)(\bar{x},\bar{u}+h(\bar{x}))(w)&=\nabla h(\bar{x})w+DS(\bar{x},\bar{u})(w),\\
D^{*}(h+S)(\bar{x},\bar{u}+h(\bar{x})(y)&=\nabla h(\bar{x})^{T}y+D^{*}S(\bar{x},\bar{u})(y),\label{sum-rule-coderivative}\\
 D_{*}(h+S)(\bar{x},\bar{u}+h(\bar{x}))(w)&=\nabla h(\bar{x})w+D_{*}S(\bar{x},\bar{u})(w),\label{sum-rule-strict-graphical-derivative}
\end{align}
which are useful in the proofs of our major results developed below.\vspace*{0.03in}

Recall now several {\em well-posedness properties} of multifunctions studied in the paper.\vspace*{-0.05in}

\begin{definition}\label{well-posed}
Let $S:\mathcal{R}^{m}\rightrightarrows\mathcal{R}^{n}$ with $\bar{x}\in\operatorname{dom}S$ and $\bar{u}\in S(\bar{x})$. Then we say that:

{\bf(i)} $S$ is {\sc metrically regular} at $\bar{x}$ for $\bar{u}$ if there exists $\kappa\ge 0$ along with $\varepsilon,\delta>0$ such that
\begin{align*}
\operatorname{dist}(x,S^{-1}(u))\leq\kappa\operatorname{dist}(u,S(x))\;\mbox{ for all }\;x\in\mathbb{B}(\bar{x},\varepsilon),\;u\in\mathbb{B}(\bar{u},\delta).
\end{align*}

{\bf(ii)} $S$ is {\sc strongly metrically regular} at $\bar{x}$ for $\bar{u}$ if it is metrically regular at $\bar{x}$ for $\bar{u}$ and there exist $\varepsilon,\delta>0$ such that $S^{-1}(u)\cap\mathbb{B}(\bar{x},\varepsilon)$ is a singleton for all $u\in\mathbb{B}(\bar{u},\delta)$.

{\bf(iii)} $S$ has the {\sc isolated calmness} property at $\bar{x}$ for $\bar{u}$ if there exist $\kappa\geq0$ and $\varepsilon,\delta>0$ with
\begin{align*}
S(x)\cap\mathbb{B}(\bar{u},\delta)\subset \{\bar{u}\}+\kappa\|x-\bar{x}\|\mathbb{B}_{\mathcal{R}^{n}}\;\mbox{ for all }\;x\in\mathbb{B}(\bar{x},\varepsilon).
\end{align*}

{\bf(iv)} $S$ is {\sc Lipschitz-like} or has the {\sc Aubin property}  at $\bar{x}$ for $\bar{u}$ with $(\bar{x},\bar{u})\in\operatorname{gph}S$ if there exist $\kappa\geq0$ and $\varepsilon,\delta>0$ ensuring that
\begin{align*}
S(x)\cap\mathbb{B}(\bar{u},\delta)\subset S(x')+\kappa\|x-x'\|\mathbb{B}_{\mathcal{R}^{n}}\;\mbox{ for all }\;x,x'\in\mathbb{B}(\bar{x},\varepsilon).
\end{align*}
\end{definition}
Relationships between the well-posedness properties in Definition~\ref{well-posed} have been clearly understood in variational analysis; see, e.g., the books \cite{DontchevandRockafellar2014,Mordukhovich2018,Rockafellar1998} and the references therein.\vspace*{0.05in}

The following three lemmas present generalized differential characterizations of the Lipschitz-like/Aubin, isolated calmness, and strong metric regularity properties that are used in this paper.

\begin{lemma}\label{lemma-aubin-property-Mordukhovich-criteria} Let the graph of $S:\mathcal{R}^{m}\rightrightarrows\mathcal{R}^{n}$ be closed around $(\bar x,\bar u)\in{\rm gph}\,S$. Then $S$ is Lipschitz-like at $\bar{x}$ for $\bar{u}$ if and only if we have the condition $D^{*}S(\bar{x},\bar{u})(0)=\{0\}$ via the limiting coderivative from \eqref{cod}.
\end{lemma}

\begin{lemma}\label{lemma-isolated-calmness-graphical-derivative} Let the graph of $S:\mathcal{R}^{m}\rightrightarrows\mathcal{R}^{n}$ be closed around $(\bar x,\bar u)\in{\rm gph}\,S$. Then $S$ enjoys the isolated calmness property at $\bar{x}$ for $\bar{u}$ if and only if we have the condition $DS(\bar{x},\bar{u})(0)=\{0\}$ in terms of the graphical derivative from \eqref{gr-der}.
\end{lemma}

\begin{lemma}\label{lemma-strong-regularity-strict-graphical-derivative} Let the graph of $S:\mathcal{R}^{m}\rightrightarrows\mathcal{R}^{n}$ be closed around $(\bar x,\bar u)\in{\rm gph}\,S$. Then $S$ is strongly metrically regular at $\bar{x}$ for $\bar{u}$ if and only if  $D^{*}S^{-1}(\bar{u},\bar{x})(0)=\{0\}$ and $D_{*}S^{-1}(\bar{u},\bar{x})(0)=\{0\}$ in terms of the strict graphical derivative from \eqref{sgr-der}.
\end{lemma}\vspace*{-0.05in}

Lemma~\ref{lemma-aubin-property-Mordukhovich-criteria} is known as the {\em Mordukhovich criterion}; see \cite[Theorem~5.7]{Mordukhovich1993} and \cite[Theorem~9.40]{Rockafellar1998}. Lemma~\ref{lemma-isolated-calmness-graphical-derivative} is called the {\em Levy-Rockafellar criterion}; see \cite[Theorem~4E.1]{DontchevandRockafellar2014} and the references therein. Lemma~\ref{lemma-strong-regularity-strict-graphical-derivative} is taken from \cite[Theorem~2.7]{Gfrerer2022}. This criterion has been recently enhanced in \cite[Theorem~4.2]{BenkoRockafellar2024} in the context where the multifunction $S$ is the solution mapping $S_{\operatorname{KKT}}$ associated with the KKT system in \eqref{eq:kkt-perturbed-cop}.

To this end, it is worth mentioning that the limiting coderivative in Lemma~\ref{lemma-aubin-property-Mordukhovich-criteria}, along with the limiting subdifferential \eqref{lsub} and normal cone in \eqref{nc}, enjoys {\em full calculus}. This is not the case for the graphical derivative \eqref{gr-der} in Lemma~\ref{lemma-isolated-calmness-graphical-derivative} and its strict counterpart \eqref{sgr-der} in Lemma~\ref{lemma-strong-regularity-strict-graphical-derivative},
as well as the regular subdifferential \eqref{rsub} and the corresponding normal cone in \eqref{nc}.\vspace*{-0.15in}

\section{Qualification and Second-Order Variational Conditions}\label{sec:srcq-nondegeneracy-sosc-ssosc}\vspace*{-0.05in}

In this section, we define and investigate basic constraint qualification, nondegeneracy, and second-order variational conditions important for deriving our main well-posedness characterization in what follows. Recall first that the {\em Robinson constraint qualification} (RCQ) holds at a feasible point $\bar x$ of the composite optimization problem \eqref{comp-prob} if
\begin{align}\label{RCQ}
\nabla F(\bar x)\mathcal{R}^{n}+\mathcal{T}_{\operatorname{dom}g}(F(\bar x))=\mathcal{R}^{m}.
\end{align}
Fixed $x\in\mathcal{R}^{n}$, define the {\em Lagrange multiplier mapping} $\Lambda$ associated with $x$ by
\begin{align}\label{lag-map}
\Lambda(x,a,b):=\big\{u\,\big|\,(a,b)\in\Phi(x,u)\big\}\ \mbox{ for }\ (a,b)\in\mathcal{R}^{n}\times\mathcal{R}^{m}.
\end{align} 
It follows from \cite[Theorem~3.9 and Proposition~3.17]{Bonnans2000} that the Lagrange multiplier set $\Lambda(x,0,0)$ is nonempty, convex, and compact if and only if RCQ holds at $x$.
If $\bar{x}$ is a local minimizer of problem \eqref{comp-prob} and RCQ is satisfied at $\bar{x}$, then there exists a Lagrange  multiplier $\bar{u}\in\mathcal{R}^{m}$ such that $(\bar{x},\bar{u})\in S_{\operatorname{KKT}}(0,0)$. It is said that the {\em strict Robinson constraint qualification} (SRCQ) is satisfied at $\bar{x}$ with respect to $\bar{u}$ if we have
\begin{align}\label{SRCQ}
\nabla F(\bar{x})\mathcal{R}^{n}+\big\{d\,\big|\,dg(F(\bar{x}))(d)=\langle\bar{u},d\rangle\big\}=\mathcal{R}^{m}.
\end{align} 
The {\em constraint nondegeneracy} condition for problem \eqref{comp-prob} is given by
\begin{align}\label{nondegeneracy-cond}
\nabla F(\bar{x})\mathcal{R}^{n}+\operatorname{lin}\big\{d\,\big|\,dg(F(\bar{x}))(d)=\langle\bar{u},d\rangle\big\}=\mathcal{R}^{m}.
\end{align}
Define further the {\em critical cone} associated with problem \eqref{comp-prob} at $\bar{x}$ by
\begin{align}\label{crit-cone}
\mathcal{C}(\bar{x})&:=\big\{d\,\big|\,\nabla h(\bar{x})^{T}d+dg(F(\bar{x}))(\nabla F(\bar{x})d)=0\big\}\\
&=\big\{d\,\big|\,dg(F(\bar{x}))(\nabla F(\bar{x})d)=\langle\nabla F(\bar{x})d,u\rangle\big\},\nonumber
\end{align}
where $u\in\Lambda(\bar{x},0,0)$.
The {\em second-order sufficient condition} (SOSC)
is said to hold for problem \eqref{comp-prob} at $\bar{x}$ if  for all $d\in\mathcal{C}(\bar{x})\setminus\{0\}$, we have
\begin{align}\label{SOSC}
\sup_{u\in\Lambda(\bar{x},0,0)}\big\{\langle\nabla_{xx}^{2}L(\bar{x},u)d,d\rangle+\inf_{w\in\mathcal{R}^{m}}\{d^{2}g(F(\bar{x}))(\nabla F(\bar{x})d,w)-\langle u,w\}\big\}>0.
\end{align}
Given now an l.s.c.\ function $f:\mathcal{R}^{m}\rightarrow\overline{\mathcal{R}}$ with $x\in\operatorname{dom}f$ and given its limiting subgradient $\mu\in\partial f(x)$, define the {\em second-order variational function} 
\begin{align}\label{Gamma}
\Gamma_{f}(x,\mu)(v):=\left\{\begin{array}{cl}
\min\limits_{{d\in\mathcal{R}^{m},\,v=Ud}\atop{U\in\mathcal{J}\operatorname{Prox}_{f}(x+\mu)}}\langle v,d-v\rangle, & v\in\bigcup\limits_{U\in\mathcal{J}\operatorname{Prox}_{f}(x+\mu)}\operatorname{rge}U,\\
\infty, & \mbox{otherwise},
\end{array}\right.
\end{align}
by using the generalized Jacobian of the proximal mapping \eqref{prox}. Following
\cite{Tangwang2024}, consider the generalized version of the strong second-order sufficient condition for the composite optimization problem \eqref{comp-prob} that extends the classical one for NLPs introduced in \cite{Robinson1980}.\vspace*{-0.03in}

\begin{definition}\label{def:ssosc}
Let $\bar{x}$ be an optimal solution to problem \eqref{comp-prob}. We say that the {\sc strong second-order sufficient condition} $(SSOSC)$ holds for \eqref{comp-prob} at $\bar{x}$ if 
\begin{align}\label{SSOSC}
\sup_{u\in\Lambda(\bar{x},0,0)}\left\{\left\langle\nabla_{xx}^{2}L(\bar{x},u)d,d\right\rangle+\Gamma_{g}(F(\bar{x}),u)(\nabla F(\bar{x})d)\right\}>0\;\mbox{ whenever }\;d\ne 0.
\end{align} 
\end{definition}\vspace*{-0.05in}

Next we formulate the two major assumptions systematically used in the paper.\vspace*{0.03in}

{\bf(A1)} Let $f$ be an l.s.c.\ function with $\bar{x}\in\operatorname{dom}f$. Given any $x$ in some neighborhood of $\bar{x}$ and any subgradient $\mu\in\partial f(x)$, the function $f$ is parabolically regular at $x$ for $\mu$ and parabolically epi-differentiable at $x$ for every $z\in\left\{d\,\left|\,df(x)(d)=\langle\mu,d\rangle\right.\right\}$.

{\bf(A2)} Let $f$ be an l.s.c.\ function with $\bar{x}\in\operatorname{dom}f$ and $\bar{\mu}\in\partial f(\bar{x})$. There exists $\overline{W}\in\partial_{B}\operatorname{Prox}_{f}(\bar{x}+\bar{\mu})$ from the $B$-subdifferential \eqref{B-sub} of the proximal mapping \eqref{prox} such that
\begin{align*}
\bigcup_{V\in\mathcal{J}\operatorname{Prox}_{f}(\bar{x}+\bar{\mu})}\operatorname{rge}V=\operatorname{rge}\overline{W}
\end{align*}
and that for any $y\in\operatorname{rge}\overline{W}$, we have $y=\overline{W}\bar{d}$, where $\bar{d}$ satisfies the equalities
\begin{align*}	         \min\limits_{{d\in\mathcal{R}^{m},y=Vd}\atop{V\in{\mathcal{J}\operatorname{Prox}_{f}(\bar{x}+\bar{\mu})}}}\langle y,d-y\rangle=\langle y,\bar{d}-y\rangle=\big\langle y,(\overline{W}^{\dagger}-I)y\big\rangle.
\end{align*}

Assumption (A1) is not restrictive at all, while (A2) holds for many broad classes of functions that frequently appear in composite optimization; see below. In practical numerical implementations, the structures involved in (A2) are often computed as a part of the solution process.

\begin{definition}\label{c2-reduc}
{\bf(i)} 
Let $C\subset\mathcal{R}^{m}$ and $K\subset\mathcal{R}^{s}$ be convex closed sets. We say that $C$ is {\sc $\mathcal{C}^2$-reducible} to $K$, at a point $x\in C$, if there exist a neighborhood $N$ at $x$ and a ${\cal C}^2$-smooth mapping $\Xi:N\rightarrow\mathcal{R}^{s}$ such that
$\nabla\Xi(x):\mathcal{R}^{m}\rightarrow\mathcal{R}^{s}$ is surjective and $C\cap N=\{x\in N\, |\, \Xi(x)\in K\}$.We say that the reduction is pointed if the tangent cone $\mathcal{T}_{K}(\Xi(x))$ is pointed. If in addition the set $K-\Xi(x)$ is a pointed closed convex cone, we say that $C$ is {\sc $\mathcal{C}^{2}$-cone reducible} at $x$. We can assume without loss of generality that $\Xi(x)=0$.

{\bf(ii)} An l.s.c.\ convex function $f:\mathcal{R}^{m}\rightarrow\overline{\mathcal{R}}$ is {\sc $\mathcal{C}^{2}$-cone reducible} at $x\in{\rm dom}\,f$ if the associated epigraphical set $\operatorname{epi}f$ is $\mathcal{C}^{2}$-cone reducible at $(x,f(x))$. Furthermore, $f$ is said to be $\mathcal{C}^{2}$-cone reducible if it is $\mathcal{C}^{2}$-cone reducible at every point $x\in\operatorname{dom}f$.
\end{definition}\vspace*{-0.05in}

It is shown in \cite[Proposition~3.1]{Tangwang2024} that {\em every $\mathcal{C}^{2}$-cone reducible function} satisfies assumption (A1). Furthermore, it is proved in Appendices of \cite{MordukhovichTangWang2025,MordukhovichTangWang-tilt-2025} that assumption (A2) holds for the indicator function of the {\em positive semidefinite cone}, the {\em nuclear norm} function, and the {\em spectral norm} function, which all are important in applications.\vspace*{0.03in}

Before establishing the main result of this section about the calculation of the second-order variational function \eqref{Gamma}, we present an important lemma of its own interest. The next lemma reveals a fundamental relationship between the second subderivative \eqref{2subder} of an extended-real-valued convex function and its {\em second-order subdifferential} introduced in \cite{Mordukhovich1992} as the limiting coderivative in \eqref{cod} of the limiting subdifferential \eqref{lsub}.\vspace*{-0.05in}

\begin{lemma}\label{lemma-second-subderivative-equality}
Let $f:\mathcal{R}^{m}\rightarrow\overline{\mathcal{R}}$ be an l.s.c.\ convex function with $\bar{x}\in\operatorname{dom}f$ and $\bar{\mu}\in\partial f(\bar{x})$ under the fulfillment of $\mathcal{C}^{2}$-cone reducibility around $\bar{x}$ and (A2) at $\bar{x}$ for $\bar{\mu}$. Then we have the equality
\begin{align}\label{leq-Gamma-f-cond}
\min_{v\in D^{*}(\partial f)(\bar{x},\bar{\mu})(d)}\langle v,d\rangle=d^{2}f(\bar{x},\bar{\mu})(d)\;\mbox{ for all }\;d\in\R^m \;
\mbox{ with }\;df(\bar{x})(d)=\langle\bar{\mu},d\rangle.
\end{align}
\end{lemma}
\begin{proof}
    It follows from assumption (A1) and \cite[Lemma~5.2]{MordukhovichTangWang2025} that
    \begin{align*}
        \operatorname{dom}d^{2}f(\bar{x},\bar{\mu})&=\operatorname{dom}D(\partial f)(\bar{x},\bar{\mu})=\{d\,|\,df(\bar{x})(d)=\langle\bar{\mu},d\rangle\}.
    \end{align*}
    Applying \cite[Proposition~3.8]{Mohammadi2020}, \cite[Lemma~3.6]{Chieu2021}, and \cite[Theorem~13.57]{Rockafellar1998} tells us that $f$ is twice epi-differentiable at $\bar{x}$ for $\bar{\mu}$ and the inequality ``$\le$" holds in \eqref{leq-Gamma-f-cond}. To verify the reverse inequality therein, let $\Xi$ and $K$ be the corresponding  $\mathcal{C}^{2}$-smooth mapping and the closed convex cone of the $\mathcal{C}^{2}$-cone reducibility set $\operatorname{epi}f$ at $(\bar{x},f(\bar{x}))$. Denote $\bar{z}:=(\bar{x},f(\bar{x}))$. It follows from \cite[Theorem~8.9]{Rockafellar1998} and \cite[Proposition~3.1]{Tangwang2024} that $(\bar{\mu},-1)\in\partial\delta_{\operatorname{epi}f}(\bar{z})$ and  $\operatorname{dom}d^{2}\delta_{\operatorname{epi}f}(\bar{z},(\bar{\mu},-1))=\{(d,\gamma)\;|\;d\in\operatorname{dom}d^{2}f(\bar{x},\bar{\mu}),\;\gamma=\langle\bar{\mu},d\rangle\}$. By \cite[Proposition~3.7]{MordukhovichSarabi2019}, \cite[Lemma~3.6]{Chieu2021}, \cite[Proposition~3.41]{Bonnans2000}, the equality \eqref{normal-cone-closed-convex-cone}, the parabolic regularity of $\delta_{\operatorname{epi}f}$ and $f$, it yields that
    \begin{eqnarray}\label{eq:second-order-subdifferential-relation}
        \begin{aligned}d^{2}\delta_{\operatorname{epi}f}(\bar{z},(\bar{\mu},-1))(d,\gamma)
        &=\langle (d,\gamma),\nabla^{2}\langle \bar{m},\Xi(\cdot)\rangle(\bar{z})(d,\gamma)\rangle\\
        &=\inf_{(w,\beta)\in\mathcal{T}^{2}_{\operatorname{epi}f}(\bar{z},(d,\gamma))}-\langle(\bar{\mu},-1),(w,\beta)\rangle\\
        &=\inf_{w\in\mathcal{R}^{m}}\{d^{2}f(\bar{x})(d,w)-\langle\bar{\mu},w\rangle\}=d^{2}f(\bar{x},\bar{\mu})(d),
        \end{aligned}
    \end{eqnarray}
    where $\bar{m}$ is the unique solution to the system 
    \begin{eqnarray*}
        (\nabla\Xi(\bar{z}))^{T}m=(\bar{\mu},-1)\quad\mbox{and}\quad m\in\mathcal{N}_{\mathcal{T}_{K}(\Xi(\bar{z}))}(0)=\mathcal{N}_{K}(\Xi(\bar{z})).
    \end{eqnarray*} 
    Applying the second-order chain rule from \cite[Theorem~3.4]{MordukhovichOutrata2001} to the composite function $\delta_{K}\circ\Xi$, it allows us to get the equality
    \begin{eqnarray*}
    \begin{aligned}
        &D^{*}(\partial\delta_{\operatorname{epi}f})(\bar{z},(\bar{\mu},-1))(d,\gamma)\\
        &=D^{*}(\partial(\delta_{K}\circ\Xi))(\bar{z},(\bar{\mu},-1))(d,\gamma)\\
        &=\nabla^{2}\langle \bar{m},\Xi(\cdot)\rangle(\bar{z})(d,\gamma)+(\nabla\Xi(\bar{z}))^{T}D^{*}(\partial\delta_{K})(\Xi(\bar{z}),\bar{m})(\nabla\Xi(\bar{z})(d,\gamma)).
        \end{aligned}
    \end{eqnarray*}
    This together with \cite[Theorem~2.1]{Poliquin1998}, \cite[Theorem~3.1]{MordukhovichTangWang-tilt-2025} and \eqref{eq:second-order-subdifferential-relation} ensures that 
    \begin{eqnarray*}
        \min_{v\in D^{*}(\partial f)(\bar{x},\bar{\mu})(d)}\langle v,d\rangle\geq d^{2}f(\bar{x},\bar{\mu})(d),\,\forall\,d\in\operatorname{dom}d^{2}f(\bar{x},\bar{\mu}).
    \end{eqnarray*}
    Then the equality  \eqref{leq-Gamma-f-cond} follows and thus completes the proof.
\end{proof}
\vspace*{-0.05in}

Now we show that the second-order variational function \eqref{Gamma} associated with $f$ reduces to the second-order subderivative \eqref{2subder} of $f$. This requires both assumptions (A1) and(A2).\vspace*{-0.05in}

\begin{theorem}\label{theorem:SOVF-second-order-subderivative}
Let $f:\mathcal{R}^{m}\rightarrow\overline{\mathcal{R}}$ be an l.s.c. convex function with $\bar{x}\in\operatorname{dom}f$ and $\bar{\mu}\in\partial f(\bar{x})$, and let both assumptions of $\mathcal{C}^{2}$-cone reducibility around $\bar{x}$ and {\rm(A2)} hold at $\bar{x}$ for $\bar{\mu}$. Then we have the identity
\begin{align}\label{second-variational-func-equality}
\Gamma_{f}(\bar{x},\bar{\mu})(d)=d^{2}f(\bar{x},\bar{\mu})(d)\;\mbox{ for all }\;d\in\operatorname{dom}d^{2}f(\bar{x},\bar{\mu}).
\end{align}
\end{theorem}
\begin{proof} It follows from \cite[Theorem~3.8]{Mohammadi2020} and \cite[Proposition~3.2]{MordukhovichTangWang-tilt-2025} that
\begin{align*}
\operatorname{dom}\Gamma_{f}(\bar{x},\bar{\mu})=\operatorname{aff}\operatorname{dom}d^{2}f(\bar{x},\bar{\mu}).
\end{align*}
Furthermore, \cite[Proposition~3.3]{MordukhovichTangWang-tilt-2025} tells that
\begin{align*}
\Gamma_{f}(\bar{x},\bar{\mu})(d)=\min_{v\in D^{*}(\partial f)(\bar{x},\bar{\mu})(d)}\left\langle v,d\right\rangle\;\mbox{ whenever }\;d\in\operatorname{dom}\Gamma_{f}(\bar{x},\bar{\mu}).
\end{align*}
Consequently, applying Lemma~\ref{lemma-second-subderivative-equality} brings us to \eqref{second-variational-func-equality} and thus completes the proof.
\end{proof}\vspace*{-0.05in}

Note that the identity \eqref{second-variational-func-equality} in Theorem~\ref{theorem:SOVF-second-order-subderivative} allows us to replace $\Gamma_f$ by the second-order subderivative in the SSOSC definition \eqref{SSOSC}, which can make it easier for computation. On the other hand, Theorem~\ref{theorem:SOVF-second-order-subderivative} leads us to the reformulation of SOSC \eqref{SOSC} in the form
\begin{align*}
\sup_{u\in\Lambda(\bar{x},0,0)}\left\{\left\langle\nabla_{xx}^{2}L(\bar{x},u)d,d\right\rangle+\Gamma_{g}(F(\bar{x}),u)(\nabla F(\bar{x})d)\right\}>0,\;\mbox{ for all }\;d\in\mathcal{C}(\bar{x})\setminus\{0\},
\end{align*} 
which highlights the relationship between SOSC and SSOSC in the general composite setting.\vspace*{0.03in} 

We conclude this section by explicit calculation of SSOSC \eqref{SSOSC} for an important class of composite optimization problems \eqref{comp-prob}, where $g:= \|\cdot\|_{2}$ is the {\em spectral norm function} associated with a given matrix $A$. Recall that $\|A\|_{2}$ is defined as the largest singular value of $A$. This function plays a pivotal role in applications, particularly to machine learning and data science. As mentioned above, both assumptions (A1) and (A2) are satisfied for \eqref{comp-prob} with $g= \|\cdot\|_{2}$. The following example provides complete calculations of SSOSC for this problem via its given data.\vspace*{-0.05in}

\begin{example}\label{exa-ssosc}
{\rm Let $p\leq q$ be two positive integers, and let $g$ be the matrix spectral norm function $\|\cdot\|_{2}$, and let $\bar{x}$ be a local minimizer of problem \eqref{comp-prob} with $u\in\Lambda(\bar{x},0,0)$. Consider the matrix $A:=F(\bar{x})+u$ whose singular value decomposition is $A=R[\Sigma\ 0]S^{T}$ with $R\in\mathcal{R}^{p\times p}$ and $S\in\mathcal{R}^{q\times q}$ being orthogonal matrices. Suppose that $S=[S_{1}\ S_{2}]$ with $S_{1}\in\mathcal{R}^{q\times p}$, $S_{2}\in\mathcal{R}^{q\times(q-p)}$ and that $\Sigma=\operatorname{Diag}(\sigma)$, where $\sigma=(\sigma_{1},\sigma_{2},\cdots,\sigma_{p})$ with $\sigma_{1}\geq\ldots\geq\sigma_{r}>0=\sigma_{r+1}=\ldots=\sigma_{p}$. 
	
Applying the result in \cite[Appendix]{MordukhovichTangWang-tilt-2025}, SSOSC \eqref{SSOSC} can be expressed as the following form according to the value of the nuclear norm $\|A\|_{*}$, which is the sum of the singular values of $A$.
	
{\bf(i)} $\|A\|_{*}<1$. Since $\mathcal{C}(\bar{x})=\{0\}$, SSOSC \eqref{SSOSC} holds trivially.
	
{\bf(ii)} $\|A\|_{*}=1$. Then $\Gamma_{g}(F(\bar{x}),u)(\nabla F(\bar{x})d)=0$ for $u\in\Lambda(\bar{x},0,0)$ and SSOSC \eqref{SSOSC} is 
\begin{align*}
&\sup_{u\in\Lambda(\bar{x},0,0)}\left\{\langle\nabla^{2}_{xx}L(\bar{x},u)d,d\rangle\right\}>0\;\mbox{ for all }\;d\ne0\;\mbox{ with}\\
&\qquad\qquad\;\nabla F(\bar{x})d\in\left\{\left.R\left[\begin{array}{cc}
aI_{r} & 0_{r\times (q-r)}\\
0_{(p-r)\times r} & D
\end{array}\right]S^{T}\,\right|\,a\in\mathcal{R},\,D\in\mathcal{R}^{(p-r)\times(q-r)}\right\}.
\end{align*}
	
{\bf(iii)} $\|A\|_{*}>1$. Let $k_{1}(\sigma)$ and $k_{2}(\sigma)$ denote the maximal indexes of the sets
\begin{align*}
\left\{i\,\left|\,\sigma_{i}>\frac{1}{i}\left(\sum_{j=1}^{i}\sigma_{j}-1\right)\right.\right\}\;\mbox{ and }\;\left\{i\,\left|\,\sigma_{i}\geq\frac{1}{i}\left(\sum_{j=1}^{i}\sigma_{j}-1\right)\right.\right\},
\end{align*}
respectively, $\mathrm{B}:=\{x\in\mathcal{R}^{p}\,|\,\|x\|_{1}\leq 1\}$ with $\operatorname{\Pi}_{\mathrm{B}}(\sigma):=(p_{1},\ldots,p_{s},0,\ldots,0)^{T}$ and $s\in[k_{1}(\sigma),k_{2}(\sigma)]$, $\nabla F(\bar{x})d:=Y$ with $R^{T}YS=(y_{ij})_{p\times q}$. Then SSOSC \eqref{SSOSC} can be written as
\begin{align*}
&\sup_{u\in\Lambda(\bar{x},0,0)}\left\{\langle\nabla_{xx}^{2}L(\bar{x},u)d,d\rangle+2\sum_{1\leq i<j\leq k_{1}(\sigma)}\frac{p_{i}+p_{j}}{\sigma_{i}+\sigma_{j}-p_{i}-p_{j}}\left(\frac{y_{ij}-y_{ji}}{2}\right)^{2}\right.\\
&\quad+2\sum_{{1\leq i\leq k_{1}(\sigma)}\atop{k_{1}(\sigma)+1\leq j\leq r}}\left[\frac{p_{i}}{\sigma_{i}-\sigma_{j}-p_{i}}\left(\frac{y_{ij}+y_{ji}}{2}\right)^{2}+\frac{p_{i}}{\sigma_{i}+\sigma_{j}-p_{i}}\left(\frac{y_{ij}-y_{ji}}{2}\right)^{2}\right]\\
&\quad+\left.2\sum_{{1\leq i\leq k_{1}(\sigma)}\atop{r+1\leq j\leq p}}\frac{p_{i}}{\sigma_{i}-p_{i}}\left[\left(\frac{y_{ij}+y_{ji}}{2}\right)^{2}+\left(\frac{y_{ij}-y_{ji}}{2}\right)^{2}\right]+\sum_{i=1}^{k_{1}(\sigma)}\sum_{j=p+1}^{q}\frac{p_{i}}{\sigma_{i}-p_{i}}y_{ij}^{2}\right\}>0\\
&\quad\mbox{ for }\;d\neq 0,\;\nabla F(\bar{x})d\in\left\{D-R[Z_{1}\ Z_{2}]S^{T}\, \left|\,Z_{1}=\Omega\circ D_{1}^{s}+\Gamma\circ D_{1}^{a}-\frac{\sum_{i=1}^{k_{1}(\sigma)}\tilde{d}_{ii}}{k_{1}(\sigma)}\left[\begin{array}{cc}
I_{k_{1}(\sigma)} & 0\\ 0 & 0
\end{array}\right],\right.\right.\\
&\quad\quad\left.Z_{2}=\left[\begin{array}{c}
\Upsilon\, 1^{T}_{q-p} \\0
\end{array}\right]\circ D_{2},\, D\in\mathcal{R}^{p\times q}\right\},
\end{align*}
where $D_{1}^{s}=\frac{D_{1}+D_{1}^{T}}{2}$, $D_{1}^{a}=\frac{D_{1}-D_{1}^{T}}{2}$, $D_{1}=R^{T}DS_{1}=(\tilde{d}_{ij})p\times p$, $D_{2}=R^{T}DS_{2}$, the matrices $\Omega,\Gamma\in\mathcal{R}^{p\times p}$ and vector $\Upsilon\in\mathcal{R}^{k_{1}(\sigma)}$ above are given, respectively, by
\begin{align*}
\Omega&:=\left[\begin{array}{ccc}
1_{k_{1}(\sigma)\times k_{1}(\sigma)} & \Omega_{\alpha_{1}\gamma} & \Omega_{\alpha_{1}\alpha_{4}} \\ \Omega_{\alpha_{1}\gamma}^{T} & 0_{(r-k_{1}(\sigma))\times(r-k_{1}(\sigma))} & 0_{(r-k_{1}(\sigma))\times(p-r)} \\
\Omega_{\alpha_{1}\alpha_{4}}^{T} & 0_{(p-r)\times(r-k_{1}(\sigma))} & 0_{(p-r)\times(p-r)}
\end{array}\right],\ \begin{array}{l}\alpha_{1}:=\{1,\ldots,k_{1}(\sigma)\},\\
\alpha_{2}:=\{k_{1}(\sigma)+1,\ldots,k_{2}(\sigma)\},\\
\alpha_{3}:=\{k_{2}(\sigma,\ldots,r)\},
\end{array}\\
\Gamma&:=\left[\begin{array}{ccc}
\Gamma_{\alpha_{1}\alpha_{1}} & \Gamma_{\alpha_{1}\gamma} & \Gamma_{\alpha_{1}\alpha_{4}}\\
\Gamma_{\alpha_{1}\gamma}^{T} & 0_{(r-k_{1}(\sigma))\times(r-k_{1}(\sigma))} & 0_{(r-k_{1}(\sigma))\times(p-k_{1}(\sigma))}\\
\Gamma_{\alpha_{1}\alpha_{4}}^{T} & 0_{(p-r)\times(r-k_{1}(\sigma))} & 0_{(p-k_{1}(\sigma))\times(p-k_{1}(\sigma))}
\end{array}\right],\begin{array}{l}
\alpha_{4}:=\{r+1,\ldots,p\},\\
\gamma:=\alpha_{2}\cup\alpha_{3},
\end{array}\\
(\Omega_{\alpha_{1}\alpha_{2}})_{kj}&:=\frac{p_{k}}{\sigma_{k}-\sigma_{j+k_{1}(\sigma)}},\quad k\in\alpha_{1},\;j\in\alpha_{2}-k_{1}(\sigma),\\
(\Omega_{\alpha_{1}\alpha_{3}})_{kj}&:=\frac{p_{k}}{\sigma_{k}-\sigma_{j+k_{2}(\sigma)}},\quad k\in\alpha_{1},\;j\in\alpha_{3}-k_{2}(\sigma),\quad
(\Omega_{\alpha_{1}\alpha_{4}})_{kj}:=\frac{p_{k}}{\sigma_{k}},\quad k\in\alpha_{1},\;j\in\alpha_{4}-r,
\end{align*}
\begin{align*}
(\Gamma_{\alpha_{1}\alpha_{1}})_{kj}&:=\frac{p_{k}+p_{j}}{\sigma_{k}+\sigma_{j}},\quad k\in\alpha_{1},\;j\in\alpha_{1},\quad		
(\Gamma_{\alpha_{1}\gamma})_{kj}:=\frac{p_{k}}{\sigma_{k}+\sigma_{j+k_{1}(\sigma)}},\quad k\in\alpha_{1},\;j\in\gamma-k_{1}(\sigma),\\
(\Gamma_{\alpha_{1}\alpha_{4}})_{kj}&:=\frac{p_{k}}{\sigma_{k}},\quad k\in\alpha_{1},\,j\in\alpha_{4}-r,\quad
\Upsilon_{k}:=\frac{p_{k}}{\sigma_{k}},\quad k=1,\ldots,k_{1}(\sigma).
\end{align*}}
\end{example}\vspace*{-0.15in}
	
\section{Characterizations of Second-Order Qualification Condition}\label{sec:characterization-soqc}\vspace*{-0.05in}

The {\em second-order qualification condition} (SOQC) in problems of composite optimization of type \eqref{comp-prob} at $\bar x$ for $\bar u$ was introduced in \cite{Mordukhovich2012} in the form
\begin{align}\label{SOQC}
\operatorname{ker}(\nabla F(\bar{x})^{T})\cap D^{*}(\partial g)(F(\bar{x}),\bar{u})(0)=\{0\},
\end{align}
where $(\bar{x},\bar{u})$ is a solution to the perturbed KKT system \eqref{eq:kkt-perturbed-cop} with $(a,b)=(0,0)$. As shown in \cite[Lemma~7.2]{Mordukhovich2023}, SOQC \eqref{SOQC} reduces to the classical linear independence constraint qualification in the case of smooth problems in nonlinear programming. The main result of this section provides complete characterizations of SOQC for \eqref{comp-prob} when the function $g$ is ${\cal C}^2$-cone reducible.\vspace*{0.03in} 

We start with the following lemma of its independent value establishing a representation for the limiting coderivative of the limiting subdifferential (i.e., the second-order subdifferential) for extended-real-valued parabolically regular functions.\vspace*{-0.03in}  

\begin{lemma}\label{prop:coderivative-equiv-partial-f-partial-normal-epif}
$f:\mathcal{R}^{m}\rightarrow\overline{\mathcal{R}}$ be an l.s.c.\ convex function with $\bar{x}\in\operatorname{dom}\,f$ and $\bar{\mu}\in\partial f(\bar{x})$ under the fulfillment of assumption {\rm(A1)}. Then we have the representation
\begin{equation}\label{cod-rep}
\begin{aligned}
D^{*}(\partial f)(\bar{x},\bar{\mu})(d)&=\big\{y\,\big|\,(y,0)\in D^{*}\mathcal{N}_{\operatorname{epi}f}((\bar{x},f(\bar{x})),(\bar{\mu},-1))(d,\beta)\big\}.
\end{aligned}
\end{equation}
\end{lemma}
\begin{proof}
It follows from \cite[Theorem~8.9]{Rockafellar1998} that for any $x\in\operatorname{dom}f$, $\mu\in\partial f(x)$, and $\lambda>0$, we get  $\lambda(\mu,-1)\in\partial \delta_{\operatorname{epi}f}(x,f(x))$. Using assumption (A1), \cite[Theorem~3.8]{Mohammadi2020}, \cite[Example~13.62]{Rockafellar1998}, and Proposition~\ref{prop-parabolically-regular-delta-epi-f} above tell us that $\delta_{\operatorname{epi}f}$ is parabolically regular at $(x,f(x))$ for $\lambda(\mu,-1)$ and parabolically epi-differentiable at $(x,f(x))$ for every $(z,\gamma)\in\operatorname{dom}d\delta_{\operatorname{epi}f}(x,f(x))$. Then it follows from Proposition~\ref{prop-parabolically-regular-delta-epi-f}, \cite[Lemma~5.2]{MordukhovichTangWang2025}, \cite[Proposition~3.2]{MordukhovichTangWang-tilt-2025}, and \cite[Theorem~13.40]{Rockafellar1998}  that we have the equalities
\begin{equation*}
D\mathcal{N}_{\operatorname{epi}f}((x,f(x)),(\mu,-1))(y,\gamma)=\frac{1}{2}\partial d^{2}\delta_{\operatorname{epi}f}((x,f(x)),(\mu,-1))(y,\gamma),
\end{equation*}\vspace*{-0.2in}
\begin{equation*}
D(\partial f)(x,\mu)(y)=\frac{1}{2}\partial d^{2}f(x,\mu)(y),
\end{equation*}\vspace*{-0.35in}
\begin{align}\label{dom-graphical-deriv-partial-f}
\operatorname{dom}D(\partial f)(x,\mu)&=\{y\,|\,df(x)(y)=\langle\mu,y\rangle\}:=\mathcal{K}(x,\mu),
\end{align}\vspace*{-0.38in}
\begin{align}\label{dom-graphical-deriv-normal-epif}
\begin{array}{ll}
\operatorname{dom}D\mathcal{N}_{\operatorname{epi}f}((x,f(x)),\lambda(\mu,-1))&=\operatorname{dom}d^{2}\delta_{\operatorname{epi}f}((x,f(x)),\lambda(\mu,-1))\\
&=\left\{(y,\gamma)\,\left|\,y\in\mathcal{K}(x,\mu),\, \gamma=\langle \mu,y\rangle\right.\right\}.
\end{array}
\end{align} 
Combining Proposition~\ref{prop-parabolically-regular-delta-epi-f}, \cite[Proposition~3.41]{Bonnans2000}, the parabolic regularity of $\delta_{\operatorname{epi}f}$ and $f$ yields
\begin{equation}\label{partial-second-order-subderivative-epi-f}
\begin{aligned}
\begin{array}{ll}
&d^{2}\delta_{\operatorname{epi}f}((x,f(x)),(\mu,-1))(y,\gamma)
=\displaystyle\inf_{(w,\beta)\in\mathcal{T}^{2}_{\operatorname{epi}f}((x,f(x)),(y,\gamma))}-\langle(\mu,-1),(w,\beta)\rangle\\
&=\displaystyle\inf_{w\in\mathcal{R}^{m}}\{d^{2}f(x)(y,w)-\langle \mu,w\rangle\}=d^{2}f(x,\mu)(y),\\ 
&\frac{1}{2}\partial d^{2}\delta_{\operatorname{epi}f}((x,f(x)),(\mu,-1))(y,\gamma)=\left\{(v,0)\,\left|\,v\in D(\partial f)(x,\mu)(y)\right.\right\}.
\end{array}
\end{aligned}
\end{equation} 
Pick any $(y,0)\in D^{*}\mathcal{N}_{\operatorname{epi}f}((\bar{x},f(\bar{x})),(\bar{\mu},-1))(d,\beta)$ and deduce from the definitions in \eqref{cod} and \eqref{nc} the existence of $((x^{k},f(x^{k})),\lambda^{k}(\mu^{k},-1))\in\operatorname{gph}\mathcal{N}_{\operatorname{epi}f}$ with $((x^{k},f(x^{k})),\lambda^{k}(\mu^{k},-1))\rightarrow((\bar{x},f(\bar{x})),(\bar{\mu},-1))$, $(d^{k},\beta^{k})\rightarrow(d,\beta)$, and $(y^{k},\gamma^{k})\in\widehat{D}^{*}\mathcal{N}_{\operatorname{epi}f}((x^{k},f(x^{k})),\lambda^{k}(\mu^{k},-1))(d^{k},\beta^{k})$ such that $(y^{k},\gamma^{k})\rightarrow(y,0)$. Applying further \cite[Proposition~8.37]{Rockafellar1998} to any pair $(y^{k},\gamma^{k})\in\widehat{D}^{*}\mathcal{N}_{\operatorname{epi}f}((x^{k},f(x^{k})),\lambda^{k}(\mu^{k},-1))(d^{k},\beta^{k})$ leads us to the estimate
\begin{align*}
&\langle(y^{k},\gamma^{k}),(\tilde{y}^{k},\tilde{\gamma}^{k})\rangle\leq\langle(d^{k},\beta^{k}) ,(\tilde{d}^{k},\tilde{\beta}^{k})\rangle\\
&\qquad\mbox{for all }\ ((\tilde{y}^{k},\tilde{\gamma}^{k}),(\tilde{d}^{k},\tilde{\beta}^{k}))\in \mathcal{T}^{k}:=\mathcal{T}_{\operatorname{gph}\mathcal{N}_{\operatorname{epi}f}}((x^{k},f(x^{k})),\lambda^{k}(\mu^{k},-1)),
\end{align*}
which can be equivalently rewritten (since ${\cal T}^k$ is a cone) as 
\begin{align*}
&\langle(y^{k},\gamma^{k}),(\tilde{y}^{k},\tilde{\gamma}^{k})\rangle\leq\langle(d^{k},\beta^{k}) ,(\tilde{d}^{k},\tilde{\beta}^{k})\rangle\;\mbox{ for all }\;((\tilde{y}^{k},\tilde{\gamma}^{k}),(\tilde{d}^{k},\tilde{\beta}^{k}))\in \mathcal{T}^{k}\cap\mathbb{B}_{\mathcal{R}^{m+1}\times\mathcal{R}^{m+1}}.
\end{align*}
It follows from \eqref{dom-graphical-deriv-normal-epif} that $\tilde{y}^{k}\in\mathcal{K}(x^{k},\mu^{k})\cap\mathbb{B}_{\mathcal{R}^{m}}$, and therefore by 
\eqref{dom-graphical-deriv-partial-f} and \eqref{partial-second-order-subderivative-epi-f}
\begin{align*}
&\big\{y\,\big|\,(y,0)\in D^{*}\mathcal{N}_{\operatorname{epi}f}((\bar{x},f(\bar{x})),(\bar{\mu},-1))(d,\beta)\big\}\\
&=\Limsup_{k\rightarrow\infty}\left\{y^{k}\,\left|\,\langle(y^{k},\gamma^{k}),(\tilde{y}^{k},\tilde{\gamma}^{k})\rangle\leq\langle(d^{k},\beta^{k}) ,(\tilde{d}^{k},\tilde{\beta}^{k})\rangle\right.\right.\;\mbox{ whenever}\\
&\qquad\left.((\tilde{y}^{k},\tilde{\gamma}^{k}),(\tilde{d}^{k},\tilde{\beta}^{k}))\in \mathcal{T}^{k}\cap\mathbb{B}_{\mathcal{R}^{m+1}\times\mathcal{R}^{m+1}},\,\gamma^{k}\rightarrow0,\,(d^{k},\beta^{k})\rightarrow(d,\beta)\right\}\\
&=\Limsup_{k\rightarrow\infty}\left\{y^{k}\,\left|\,\langle y^{k},\tilde{y}^{k}\rangle-\langle d^{k},\tilde{d}^{k}\rangle\leq \varepsilon^{k}\;\mbox{ when }\;(\tilde{y}^{k},\tilde{d}^{k})\in\mathcal{T}_{\operatorname{gph}\partial f}(x^{k},\mu^{k})\cap\mathbb{B}_{\mathcal{R}^{m}},\,\varepsilon^{k}\rightarrow0,\,d^{k}\rightarrow d\right.\right\}\\
&=\Limsup_{k\rightarrow\infty}\widehat{D}^{*}(\partial f)(x^{k},\mu^{k})(d^{k})\\
&=D^{*}(\partial f)(\bar{x},\bar{\mu})(d),
\end{align*}
where ``$\Limsup$" signifies the limits of convergent subsequences (Painlev\'e-Kuratowski outer limit); cf.\ \eqref{nc} and \eqref{cod}. 
This verifies  \eqref{cod-rep} and thus completes the proof of the lemma.
\end{proof}\vspace*{-0.05in}

The next lemma elucidates interconnections between SOQC \eqref{SOQC}, the constraint nondegeneracy \eqref{nondegeneracy-cond} for the original problem 
\eqref{comp-prob}, and their counterparts in the auxiliary problem 
\begin{align}\label{eq:problem-equivalent}
\min_{(x,c)\in\mathcal{R}^{n}\times\mathcal{R}}\big\{h(x)+c\,\big|\,(F(x),c)\in\operatorname{epi}g\big\}.
\end{align}

\begin{lemma}\label{lemma:equivalent-second-order-QC}
Let $(\bar{x},\bar{u})$ be a solution to the KKT system \eqref{eq:kkt-perturbed-cop} with $(a,b)=(0,0)$. Then we have the following equivalency assertions:
	
{\bf(i)} The constraint nondegeneracy condition for \eqref{eq:problem-equivalent} is equivalent to that for problem \eqref{comp-prob}.

{\bf(ii)} If $g$ satisfies assumption {\rm(A1)} at $\bar{x}$, then SOQC \eqref{SOQC} for the original problem \eqref{comp-prob} is equivalent to the SOQC version for the auxiliary problem \eqref{eq:problem-equivalent} written as
\begin{align}\label{SOQC-equiv-prob}
\operatorname{ker}\left[\begin{array}{cc}
\nabla F(\bar{x})^{T} & 0\\
0 & 1
\end{array}\right]\cap D^{*}\mathcal{N}_{\operatorname{epi}g}((F(\bar{x}),g(F(\bar{x}))),(\bar{u},-1))(0,0)=\{0\}.
\end{align}
\end{lemma}
\begin{proof}
The constraint nondegeneracy condition for \eqref{eq:problem-equivalent} at $(\bar{x},g(F(\bar{x})))$ is
\begin{align}\label{eq:nondegenerate-cond-equivalent-problem}
\left[\begin{array}{cc}
\nabla F(\bar{x}) & 0\\
0 & 1
\end{array}\right]\mathcal{R}^{n+1}+\operatorname{lin}\mathcal{T}_{\operatorname{epi}g}(F(\bar{x}),g(F(\bar{x})))=\mathcal{R}^{m}\times\mathcal{R},
\end{align}
where the tangent cone to the epigraph is given by
\begin{align*}
\mathcal{T}_{\operatorname{epi}g}(F(\bar{x}),g(F(\bar{x})))=\left\{(d,d_{c})\,\left|\,dg(F(\bar{x}))(d)\leq d_{c}\right.\right\}.
\end{align*}
Applying \cite[Corollary~4.7.2]{Rockafellar1970} tells us that
\begin{align*}
\operatorname{lin}\mathcal{T}_{\operatorname{epi}g}(F(\bar{x}),g(F(\bar{x})))&=\left\{(d,d_{c})\,\left|\,dg(F(\bar{x}))(d)\leq d_{c}\leq-dg(F(\bar{x}))(-d)\right.\right\}\\
&=\left\{(d,d_{c})\,\left|\,dg(F(\bar{x}))(d)=-dg(F(\bar{x}))(-d)=d_{c}\right.\right\},
\end{align*}
and therefore the constraint nondegeneracy condition \eqref{eq:nondegenerate-cond-equivalent-problem} is equivalent to \eqref{nondegeneracy-cond}.
	
To verify the second assertion, observe that $(\bar{x},g(F(\bar{x})))$ solves the unconstrained problem \eqref{eq:problem-equivalent} with $(\bar{u},-1)$ being the corresponding Lagrange multiplier. Since
\begin{align*}
\operatorname{ker}\left[\begin{array}{cc}
\nabla F(\bar{x})^{T} & 0\\
0 & 1
\end{array}\right]=\left\{(d,0)\,\left|\,d\in\operatorname{ker}(\nabla F(\bar{x})^{T})\right.\right\},
\end{align*}
we represent the set intersection in \eqref{SOQC-equiv-prob} as follows:
\begin{align*}
&\operatorname{ker}\left[\begin{array}{cc}
\nabla F(\bar{x})^{T} & 0\\
0 & 1
\end{array}\right]\cap D^{*}\mathcal{N}_{\operatorname{epi}g}((F(\bar{x}),g(F(\bar{x}))),(\bar{u},-1))(0,0)\\
&=\left\{(d,0)\,\left|\,(d,0)\in D^{*}\mathcal{N}_{\operatorname{epi}g}((F(\bar{x}),g(F(\bar{x}))),(\bar{u},-1))(0,0),\,d\in\operatorname{ker}(\nabla F(\bar{x})^{T})\right.\right\}.
\end{align*}
Combining the latter with Lemma~\ref{prop:coderivative-equiv-partial-f-partial-normal-epif} justifies (ii) and thus completes the proof.
\end{proof}

Based on the above lemmas, we now establish relationships between the second-order qualification condition and constraint nondegeneracy.\vspace*{-0.03in}

\begin{theorem}\label{theorem-SOQC-nondegenerate-cond}
Let $(\bar{x},\bar{u})$ solve the KKT system \eqref{eq:kkt-perturbed-cop} with $(a,b)=(0,0)$. If the constraint nondegeneracy condition \eqref{nondegeneracy-cond} holds at $\bar{x}$ for $\bar{u}$ and in addition $g$ is $\mathcal{C}^{2}$-cone reducible, then SOQC \eqref{SOQC} is satisfied at $\bar{x}$ for $\bar{u}$.
\end{theorem}\vspace*{-0.15in}
\begin{proof}
This follows from  Lemma~\ref{lemma:equivalent-second-order-QC} combined with \cite[Corollary~3.8]{MordukhovichSarabi2015}.
\end{proof}
\vspace*{-0.25in}

\section{Characterizations of Lipschitzian Stability for KKT Systems}\label{sec:equivalence-aubin-property}\vspace*{-0.05in} 

This section contains the main results of the paper establishing close relationships and comprehensive characterizations for various Lipschitzian properties of perturbed KKT systems associated with problems of composite optimization. Our primary attention concerns the Lipschitz-like/Aubin property of the KKT solution mapping $S_{\operatorname{KKT}}$ for system \eqref{eq:kkt-perturbed-cop}.\vspace*{0.03in}

We begin with the foundational lemma that bridges the concept of isolated calmness with the combination of SRCQ and SOSC. This lemma serves as a crucial building block for our subsequent analysis of the Aubin property for $S_{\operatorname{KKT}}$.\vspace*{-0.05in}

\begin{lemma}\label{lemma-isolated-calm-SRC-SOSC}
Let $x$ be a local minimizer of problem \eqref{comp-prob-perturb} with parameters {$(a,0)$} and $u\in\Lambda(x,a,0)$ be a corresponding Lagrange multiplier. Suppose that $g:\mathcal{R}^m\to\overline{\mathcal{R}}$ satisfies  assumption {\rm(A1)}  at $F(x)$. Then the mapping $S_{\operatorname{KKT}}$ has the isolated calmness property at {$(a,0)$} for $(x,u)$ if and only if both SRCQ \eqref{SRCQ} and the following version
\begin{equation}\label{SOSC-equiv-single-dual}
\left\langle\nabla_{xx}^{2}L(x,u)d,d\right\rangle+d^{2}g(F(x),u)(\nabla F(x)d)>0\;\mbox{ for all }\;d\in\mathcal{C}(x)\setminus\{0\}
\end{equation} 
of the SOSC property hold, where ${\cal C}(x)$ is the critical cone \eqref{crit-cone} at $x$,
\end{lemma}
\begin{proof}
Fix $a\in\mathcal{R}^{n}$ and consider the modified KKT mapping
\begin{equation*}
\tilde{\Phi}((x,c),(u,\gamma)):=\left[\begin{array}{c}
\nabla_{x}L(x,u)-a\\
1+\gamma\\
-(F(x),c)
\end{array}\right]+\left[\begin{array}{c}
0\\0\\
\mathcal{N}^{-1}_{\operatorname{epi}g}(u,\gamma)
\end{array}\right]
\end{equation*} 
together with the corresponding solution mapping 
\begin{equation*}
\tilde{S}_{\operatorname{KKT}}((\tilde{a},\tilde{a}'),(\tilde{b},\tilde{b}')):=\big\{((x,c),(u,\gamma))\,\big|\,((\tilde{a},\tilde{a}'),(\tilde{b},\tilde{b}'))\in\tilde{\Phi}((x,c),(u,\gamma))\big\}.
\end{equation*}
It follows from \cite[Theorem~8.9]{Rockafellar1998} that $\tilde{S}_{\operatorname{KKT}}((0,0),(0,0))=\{((x,g(F(x))),(u,-1))\,|\,(x,u)\in S_{\operatorname{KKT}}(a,0)\}$, and hence $\tilde{S}_{\operatorname{KKT}}((0,0),(0,0))$ can be seen as the set of solutions to the KKT system associated with the canonically perturbed constrained optimization problem given by
\begin{equation}\label{canonical-perturbed-constrained-prob}
\min_{(x,c)\in\mathcal{R}^{n}\times\mathcal{R}}\left\{h(x)+c-\langle a,x\rangle\,\left|\,(F(x),c)\in\operatorname{epi}g\right.\right\}.
\end{equation}
The graphical derivative of $\tilde{\Phi}$ at $((x,c),(u,-1))$ for $((0,0),(0,0))$ is calculated by
\begin{equation*}
\begin{aligned}
&D\tilde{\Phi}(((x,c),(u,-1)),((0,0),(0,0)))((\Delta\xi,\Delta\xi'),(\Delta\eta,\Delta\eta'))\\
&=\left[\begin{array}{c}
\nabla_{xx}^{2}L(x,u)\Delta\xi+\nabla F(x)^{T}\Delta\eta\\
\Delta\eta'\\
-(\nabla F(x)\Delta\xi,\Delta\xi')+D\mathcal{N}^{-1}_{\operatorname{epi}g}((u,-1),(F(x),c))(\Delta\eta,\Delta\eta')
\end{array}\right].
\end{aligned}
\end{equation*} 
Proposition~\ref{prop-parabolically-regular-delta-epi-f}, \cite[Lemma~5.2]{MordukhovichTangWang2025}, \cite[Proposition~3.2]{MordukhovichTangWang-tilt-2025}, and \cite[Theorem~13.40]{Rockafellar1998} ensure that
\begin{equation*}
\begin{aligned}
D\mathcal{N}_{\operatorname{epi}g}((F(x),c),(u,-1))(\nabla F(x)\Delta\xi,\Delta\xi')&=\frac{1}{2}\partial d^{2}\delta_{\operatorname{epi}g}((F(x),c),(u,-1))(\nabla F(x)\Delta\xi,\Delta\xi'),\\
D(\partial g)(F(x),u)(\nabla F(x)\Delta\xi)&=\frac{1}{2}\partial d^{2}g(F(x),u)(\nabla F(x)\Delta\xi),\\
\operatorname{dom}d^{2}g(F(x),u)&=\{d\,|\,dg(F(x))(d)=\langle u,d\rangle\},\\
\operatorname{dom}d^{2}\delta_{\operatorname{epi}g}((F(x),c),(u,-1))&=\{(d,\gamma)\,|\,dg(F(x))(d)=\langle u,d\rangle=\gamma\},
\end{aligned}
\end{equation*}
where the last equality follows by applying \eqref{dom-graphical-deriv-normal-epif} with $\lambda=1$.
By \eqref{inverse-graphical-derivative}, the inclusion $((0,0),(0,0))\in D\tilde{\Phi}(((x,c),(u,-1)),((0,0),(0,0)))((\Delta\xi,\Delta\xi'),(\Delta\eta,\Delta\eta'))$ can be rewritten as
\begin{equation*}
\left[\begin{array}{c}
0\\
0\\
(0,0)\\
\end{array}\right]\in\left[\begin{array}{c}\nabla_{xx}^{2}L(x,u)\Delta\xi+\nabla F(x)^{T}\Delta\eta\\
\Delta\eta'\\
-(\Delta\eta,\Delta\eta')
\end{array}\right]+\left[\begin{array}{c}
0\\
0\\
\frac{1}{2}\partial d^{2}\delta_{\operatorname{epi}g}((F(x),c),(u,-1))(\nabla F(x)\Delta\xi,\Delta\xi')
\end{array}\right].
\end{equation*}
Combining Proposition~\ref{prop-parabolically-regular-delta-epi-f}, \cite[Proposition~3.41]{Bonnans2000}, the parabolic regularity of $\delta_{\operatorname{epi}g}$ and $g$ yields
\begin{equation*}
\begin{aligned}
\begin{array}{ll}
&d^{2}\delta_{\operatorname{epi}g}((F(x),c),(u,-1))(\nabla F(x)\Delta\xi,\Delta\xi')
=\displaystyle\inf_{(w,\beta)\in\mathcal{T}^{2}_{\operatorname{epi}g}((F(x),c),(\nabla F(x)\Delta\xi,\Delta\xi'))}-\langle(u,-1),(w,\beta)\rangle\\
&=\displaystyle\inf_{w\in\mathcal{R}^{m}}\{d^{2}g(F(x))(\nabla F(x)\Delta\xi,w)-\langle u,w\rangle\}=d^{2}g(F(x),u)(\nabla F(x)\Delta\xi),\\ 
&\frac{1}{2}\partial d^{2}\delta_{\operatorname{epi}g}((F(x),c),(u,-1))(\nabla F(x)\Delta\xi,\Delta\xi')=\left\{(v,0)\,\left|\,v\in D(\partial g)(F(x),u)(\nabla F(x)\Delta\xi)\right.\right\}.
\end{array}
\end{aligned}
\end{equation*} 
Therefore, the inclusion $((0,0),(0,0))\in D\tilde{\Phi}(((x,c),(u,-1)),((0,0),(0,0)))((\Delta\xi,\Delta\xi'),(\Delta\eta,\Delta\eta'))$ can be rewritten in the equivalent form
\begin{equation*}
\begin{aligned}
\left[\begin{array}{c}
0\\0
\end{array}\right]&\in\left[\begin{array}{c}
\nabla_{xx}^{2}L(x,u)\Delta\xi+\nabla F(x)^{T}\Delta\eta\\
-\Delta\eta+D(\partial g)(F(x),u)(\nabla F(x)\Delta\xi)
\end{array}\right]\Longleftrightarrow \left[\begin{array}{c}
0\\0
\end{array}\right]\in\left[\begin{array}{c}
\nabla_{xx}^{2}L(x,u)\Delta\xi+\nabla F(x)^{T}\Delta\eta\\
-\nabla F(x)\Delta\xi+D(\partial g^{*})(u,F(x))(\Delta\eta)
\end{array}\right].
\end{aligned}
\end{equation*}
It follows from Lemma~\ref{lemma-isolated-calmness-graphical-derivative} that the solution mapping $S_{\operatorname{KKT}}$ enjoys the isolated calmness property at $(a,0)$ for $(x,u)$ if and only if the modified KKT mapping $\tilde{S}_{\operatorname{KKT}}$ has this property at $((0,0),(0,0))$ for $((x,c),(u,-1))$.  By the above expression of $d^{2}\delta_{\operatorname{epi}g}$, the original SOSC \eqref{SOSC} for \eqref{canonical-perturbed-constrained-prob} is equivalent to the SOSC version in \eqref{SOSC-equiv-single-dual}. The equivalence between SRCQ for \eqref{canonical-perturbed-constrained-prob} and the original SRCQ \eqref{SRCQ} is trivial. Due to the fact that $(x,g(F(x)))$ is also the unique local minimizer of problem \eqref{canonical-perturbed-constrained-prob}, the proof is completed by applying \cite[Theorem~4.1]{MohammadiMordukhovichSarabi2020}.
\end{proof}\vspace*{-0.05in}

Recall yet another Lipschitzian stability property for local minimizers of extended-real-valued functions $\varphi\colon\mathcal{R}^n\rightarrow\overline{\mathcal{R}}$ in \eqref{comp-prob}. Following \cite{Poliquin1998}, we say that $\bar x\in{\rm dom}\,\varphi$ is a {\em tilt-stable} local minimizer of $\varphi$ with modulus $\kappa>0$ if there exists a number $\gamma>0$ such that the mapping
\begin{align*}
M_{\gamma}:u\mapsto\mathop{\operatorname{arg}\min}\left\{\varphi(x)-\langle u,x\rangle\,\left|\,x\in\mathbb{B}(\bar{x},\gamma)\right.\right\}
\end{align*}
is single-valued and Lipschitz continuous with modulus $\kappa$ on some neighborhood of the origin with $M_{\gamma}(0)=\{\bar{x}\}$. Over the years, this property has been largely studied and applied in variational analysis and optimization; see, e.g., the recent book \cite{Mordukhovich2024} for major achievements in this direction. The following observation tells us that the tilt stability of a local minimizer $\bar x$ for a parabolically regular function $\varphi$ yields the {\em second-order subdifferential condition} \eqref{second-order-subdifferential-condition-single-dual} when the corresponding multiplier sets around $\bar{x}$ are singletons.\vspace*{-0.05in}

\begin{lemma}\label{lemma-tilt-stability-second-order-subdifferential-cond}
Let $\bar{x}$ be a local minimizer of problem \eqref{comp-prob} with $\bar{u}\in\Lambda(\bar{x},0,0)$ being a corresponding Lagrange multiplier.  Suppose that RCQ \eqref{RCQ} holds at $\bar{x}$, that $g$ satisfies {\rm (A1)} at $F(\bar{x})$, and that there exist neighborhoods $\mathcal{U}$ of $\bar{x}$ and $\mathcal{V}$ of the origin such that  $\Lambda(x,a,0)$ is a singleton for $x\in\mathcal{U}$, $a\in\mathcal{V}$, and $(x,u)\in S_{\operatorname{KKT}}(a,0)$. If $\bar{x}$ is a tilt-stable local minimizer of $\varphi$ with modulus $\kappa>0$, then we have the second-order condition
\begin{align}\label{second-order-subdifferential-condition-single-dual}
\left\langle\nabla_{xx}^{2}L(\bar{x},\bar{u})d,d\right\rangle+\min_{z\in D^{*}(\partial g)(F(\bar{x}),\bar{u})(\nabla F(\bar{x})d)}\left\langle\nabla F(\bar{x})d,z\right\rangle\geq\frac{1}{\kappa}\|d\|^{2}\;\mbox{ for all }\;d\ne 0.
\end{align}
\end{lemma}\vspace*{-0.05in} 
\begin{proof}
Arguing  by contradiction, suppose that there exists $0\neq d\in\mathcal{R}^{n}$ with $\nabla F(\bar{x})d\in\operatorname{dom}D^{*}(\partial g)(F(\bar{x}),\bar{u})$ such that \eqref{second-order-subdifferential-condition-single-dual} fails, i.e.,
\begin{align}\label{contradiction-second-order-subdifferential-cond}
\left\langle\nabla_{xx}^{2}L(\bar{x},\bar{u})d,d\right\rangle+\min_{z\in D^{*}(\partial g)(F(\bar{x}),\bar{u})(\nabla F(\bar{x})d)}\left\langle\nabla F(\bar{x})d,z\right\rangle<\frac{1}{\kappa}\|d\|^{2}.
\end{align}
Take $\bar{z}\in D^{*}(\partial g)(F(\bar{x}),\bar{u})(\nabla F(\bar{x})d)$ with $\min\limits_{z\in D^{*}(\partial g)(F(\bar{x}),\bar{u})(\nabla F(\bar{x})d)}\left\langle\nabla F(\bar{x})d,z\right\rangle=\left\langle\bar{z},\nabla F(\bar{x})d\right\rangle$ and find sequences $(x^{k},u^{k})\rightarrow(\bar{x},\bar{u})$ with $u^{k}\in\partial g(F(x^{k}))$ and $(z^{k},d^{k})\rightarrow(\bar{z},d)$ with $z^{k}\in\widehat{D}^{*}(\partial g)(F(x^{k}),u^{k})(\nabla F(x^{k})d^{k})$ satisfying $\nabla h(x^{k})+\nabla F(x^{k})^{T}u^{k}\in\mathcal{V}$  and $\Lambda(x^{k},\nabla h(x^{k})+\nabla F(x^{k})^{T}u^{k},0)=\{u^{k}\}$ when $k$ is sufficiently large.
It follows from \cite[Lemma~3.1]{MordukhovichTangWang-tilt-2025}, \eqref{inverse-graphical-derivative}, \eqref{sum-rule-graphical-derivative}, and \cite[Lemma~5.2]{MordukhovichTangWang2025} that $-\nabla F(x^{k})d^{k}\in\operatorname{dom}D(\partial g)(F(x^{k}),u^{k})$. Combining \cite[Lemma~4.1]{MordukhovichTangWang-tilt-2025} and \cite[Theorem 3.3]{Chieu2018} brings us to the inequality
\begin{align*}
\left\langle\nabla_{xx}^{2}L(x^{k},u^{k})d^{k},d^{k}\right\rangle+d^{2}g(F(x^{k}),u^{k})(-\nabla F(x^{k})d^{k})\geq\frac{1}{\kappa}\|d^{k}\|^{2}.
\end{align*}
Using \cite[Proposition~8.37]{Rockafellar1998} and \cite[Lemma~3.6]{Chieu2021} implies that
\begin{align*}
\big\langle z^{k},\nabla F(x^{k})d^{k}\big\rangle\geq\big\langle w^{k},-\nabla F(x^{k})d^{k}\big\rangle=d^{2}g(F(x^{k}),u^{k})(-\nabla F(x^{k})d^{k}), 
\end{align*} 
where $w^{k}\in D(\partial g)(F(x^{k}),u^{k})(-\nabla F(x^{k})d^{k})$. Passing to the limit as $k\rightarrow\infty$, we arrive at a contradiction with \eqref{contradiction-second-order-subdifferential-cond} and thus complete the proof of the lemma.
\end{proof}\vspace*{-0.05in}

To proceed further, let us define some mappings associated  with the perturbed problem \eqref{comp-prob-perturb} and its KKT system \eqref{eq:kkt-perturbed-cop} via the set of Lagrange multipliers \eqref{lag-map} by
\begin{align*}
\overline{M}_{\delta}(a,b)&:=\mathop{\operatorname{arg}\min}_{\|x-\bar{x}\|\leq\delta}\big\{\tilde{\varphi}(x,b)-\langle a,x\rangle\big\}\quad \mbox{with}\ \bar{x}\in\overline{M}_{\delta}(\bar{a},\bar{b}),\\
\overline{m}_{\delta}(a,b)&:=\min_{\|x-\bar{x}\|\leq\delta}\big\{\tilde{\varphi}(x,b)-\langle a,x\rangle\big\},\\
\widehat{M}_{\delta}(a,b)&:=\big\{(x,u)\,\big|\,x\in\overline{M}_{\delta}(a,b),\, u\in\Lambda(x,a,b),\,\|u-\bar{u}\|\leq\delta\big\}\; \mbox{ with }\;\bar{u}\in\Lambda(\bar{x},\bar{a},\bar{b}),\, \delta>0.
\end{align*}

Using the above constructions, we recall the fundamental properties of (Lipschitzian) {\em full stability}. The first part of the following definition concerns {\em primal full stability} for local minimizers of the original problem \eqref{comp-prob} under perturbations in \eqref{comp-prob-perturb} introduced in \cite{LevyPoliquinRockafellar2000}, while the second notion of {\em primal-dual full stability}, which has been recently introduced in \cite{BenkoRockafellar2024}, addresses behavior of solution-multiplier pairs satisfying the associated KKT system \eqref{eq:kkt-perturbed-cop} under perturbations.\vspace*{-0.05in} 

\begin{definition}\label{full stability} {\bf(i)} A given {\sc local minimizer} $\bar{x}$ of problem \eqref{comp-prob} is {\sc fully stable} with respect to perturbation in \eqref{comp-prob-perturb} if there exists a neighborhood $\mathcal{U}\times\mathcal{V}$ of $(\bar{a},\bar{b})=(0,0)$ such that for all $\delta>0$ sufficiently small, the mapping $\overline{M}_{\delta}$ is single-valued and Lipschitz continuous on $\mathcal{U}\times\mathcal{V}$, and the function $\overline{m}_{\delta}$ is Lipschitz continuous on $\mathcal{U}\times\mathcal{V}$.

{\bf(ii)} The {\sc primal-dual solution} $(\bar{x},\bar{u})$ to the KKT system \eqref{eq:kkt-perturbed-cop} with $(a,b)=(0,0)$ is {\sc fully stable} if there exists a neighborhood $\mathcal{U}\times\mathcal{V}$ of $(\bar{a},\bar{b})=(0,0)$ such that for any $\delta>0$ sufficiently small, the mapping $\widehat{M}_{\delta}$ is
single-valued and Lipschitz continuous on $\mathcal{U}\times\mathcal{V}$, and the function $\overline{m}_{\delta}$ is Lipschitz continuous on $\mathcal{U}\times\mathcal{V}$.
\end{definition}

We refer the reader to \cite[Chapter~5]{Mordukhovich2024} for the state-of-the-art in the study of full stability of local minimizers from Definition~\ref{full stability}(i) and to \cite[Chapter~6]{Mordukhovich2024} for full stability of parametric variational systems related to Definition~\ref{full stability}(ii).\vspace*{0.05in}

In the proof of the main characterizations established below, we need one more lemma, which justifies the following {\em robustness property} of SOQC.\vspace*{-0.05in}

\begin{lemma}\label{lemma-rebust-SOQC}
Let $(\bar{x},\bar{u})$ be a solution to the KKT system \eqref{eq:kkt-perturbed-cop} with $(a,b)=(0,0)$. Suppose that SOQC \eqref{SOQC} holds at $\bar{x}$ for $\bar{u}$. Then there exist neighborhoods $\mathcal{U}$ of $\bar{x}$ and $\mathcal{V}$ of the origin such that SOQC is satisfied at $x$ for $u$  whenever $x\in\mathcal{U}$, $\delta\in\mathcal{V}$ with $(x,u)\in 
S_{\operatorname{KKT}}(\delta,0)$. \vspace*{-0.05in}
\end{lemma}
\begin{proof}
Assuming the contrary, for any $k$ sufficiently large take a sequence $\delta^{k}\in\frac{1}{k}\mathbb{B}_{\mathcal{R}^{n}}$ and find $(x^{k},u^{k})\in S_{\operatorname{KKT}}(\delta^{k},0)$ with $x^{k}\rightarrow\bar{x}$ and $w^{k}\neq 0$ with $\|w^{k}\|=1$ such that 
\begin{align}\label{contradiction-SOQC}
w^{k}\in\operatorname{ker}(\nabla F(x^{k})^{T})\cap D^{*}(\partial g)(F(x^{k}),u^{k})(0).
\end{align}
It is easy to deduce from the negation of SOQC that both $\{u^{k}\}$ and $\{w^{k}\}$ are bounded, and hence we can suppose without loss of generality that $u^{k}\rightarrow\bar{u}$ and $w^{k}\rightarrow\bar{w}$ with $\|\bar{w}\|=1$. Passing to the limit in \eqref{contradiction-SOQC} as $k\rightarrow\infty$ tells us that 
\begin{align*}
\bar{w}\in\operatorname{ker}(\nabla F(\bar{x})^{T})\cap D^{*}(\partial g)(F(\bar{x}),\bar{u})(0),
\end{align*} 
a contradiction, which verifies the claimed robustness property.
\end{proof}\vspace*{-0.05in}

Now we are ready to establish the first set of equivalences between the major well-posedness properties in composite optimization with their {\em pointbased} second-order characterizations.\vspace*{-0.05in}

\begin{theorem}\label{theorem-aubin-property-SOQC-SSOSC}
Let $(\bar{x},\bar{u})$ be a solution to the KKT system \eqref{eq:kkt-perturbed-cop} with $(a,b)=(0,0)$. Then the following assertions are equivalent:
    
{\bf{(i)}} The KKT solution mapping $S_{\operatorname{KKT}}$ has the Aubin property at the origin for $(\bar{x},\bar{u})$. 
    
{\bf(ii)} SOQC \eqref{SOQC} holds together with the second-order kernel condition at $\bar{x}$ for $\bar{u}$, i.e.,
\begin{equation}\label{aubin-property-equiv-cond}
\left\{d\,|\,0\in\nabla^{2}_{xx}L(\bar{x},\bar{u})d+\nabla F(\bar{x})^{T}D^{*}(\partial g)(F(\bar{x}),\bar{u})(\nabla F(\bar{x})d)\right\}=\{0\}.
\end{equation}
If $\bar{x}$ is a local minimizer of \eqref{comp-prob} and $(A1)$ is fulfilled for $g$ at $F(\bar{x})$, the previous statements are also equivalent to:\\

{\bf(iii)} Both SOQC \eqref{SOQC} and the second-order subdifferential condition \eqref{second-order-subdifferential-condition-single-dual} are satisfied at $\bar{x}$.

{\bf(iv)} SOQC \eqref{SOQC} holds, and $\bar{x}$ is a tilt-stable local minimizer of $\varphi$.\\
Moreover, if $g$ is $\mathcal{C}^{2}$-cone reducible at $F(\bar x)$, then the list of equivalences is extended to: 

{\bf(v)} The pair $(\bar{x},\bar{u})$ is a strongly regular point of the KKT system \eqref{eq:kkt-perturbed-cop}.

{\bf(vi)} The constraint nondegeneracy condition \eqref{nondegeneracy-cond} holds at $\bar{x}$ for $\bar{u}$ and $\bar{x}$ is a fully stable local minimizer of problem \eqref{comp-prob-perturb}.
	
{\bf(vii)} The solution $(\bar{x},\bar{u})$ to the KKT system \eqref{eq:kkt-perturbed-cop} is fully stable with respect to the origin.
	
{\bf(viii)} The KKT mapping $\Phi$ from \eqref{Phi} is strongly metrically regular at $(\bar{x},\bar{u})$ for $(0,0)$.\vspace*{-0.05in}
\end{theorem}
\begin{proof}
First we verify equivalence (i)$\Leftrightarrow$(ii). Given $x\in\operatorname{dom}(g\circ F)$, $u\in\partial g(F(x))$, and $(\delta,0)\in\Phi(x,u)$, we prove that the Aubin property of $S_{\operatorname{KKT}}$ at $(\delta,0)$ for $(x,u)$ is equivalent to SOQC at $x$ for $u$ combined with the second-order kernel condition at $x$ for $u$. The coderivative of $\Phi$ at $(x,u)$ for $(\delta,0)$ in the direction $(\Delta\xi,\Delta\eta)$ is represented by
\begin{equation}\label{cod-rep1}
\begin{array}{ll}
D^{*}\Phi((x,u),(\delta,0))(\Delta\xi,\Delta\eta)&=\left[\begin{array}{c}
\nabla^{2}_{xx}L(x,u)\Delta\xi-\nabla F(x)^{T}\Delta\eta\\\nabla F(x)\Delta\xi
\end{array}\right]+\left[\begin{array}{c}
0\\I
\end{array}\right] D^{*}(\partial g^{*})(u,F(x))(\Delta\eta).
\end{array}
\end{equation}
By the Mordukhovich criterion in Lemma~\ref{lemma-aubin-property-Mordukhovich-criteria}, the Aubin property of $S_{\operatorname{KKT}}$ at $(\delta,0)$ for $(x,u)$ is equivalent to the implication $(0,0)\in D^{*}\Phi((x,u),(\delta,0))(\Delta\xi,\Delta\eta)\Longrightarrow(\Delta\xi,\Delta\eta)=(0,0)$. This tells us by \eqref{cod-rep1}
that if $\Delta\zeta\in D^{*}(\partial g^{*})(u,F(x))(\Delta\eta)$ with
\begin{equation}\label{eq:coderivative-Phi}
\nabla^{2}_{xx}L(x,u)\Delta\xi-\nabla F(x)^{T}\Delta\eta=0\;\mbox{  and }\;\nabla F(x)\Delta\xi+\Delta\zeta=0,
\end{equation}
then $\Delta\xi=0$ and $\Delta\eta=0$. Since $-\Delta\eta\in D^{*}(\partial g)(F(x),u)(\nabla F(x)\Delta\xi)$, we get from the first equation of \eqref{eq:coderivative-Phi} that the vector $\Delta\xi=0$
solves the inclusion
\begin{equation*}
0\in\nabla_{xx}^{2}L(x,u)\Delta\xi+\nabla F(x)^{T}D^{*}(\partial g)(F(x),u)(\nabla F(x)\Delta\xi).
\end{equation*}
This allows us to deduce from the equivalence
\begin{equation*}
0\in D^{*}\Phi((x,u),(\delta,0))(0,\Delta\eta)\ \Longleftrightarrow\ \nabla F(x)^{T}\Delta\eta=0,\ -\Delta\eta\in D^{*}(\partial g)(F(x),u)(0)
\end{equation*}
that $\Delta\eta=0$. The reverse implication is easy. Indeed, the second-order kernel condition gives us $\Delta\xi=0$ in \eqref{eq:coderivative-Phi}, which ensures by SOQC at $x$ for $u$ that $\Delta\eta=0$. Then the Aubin property follows directly from the characterization in Lemma~\ref{lemma-aubin-property-Mordukhovich-criteria}.
Letting $(x,u):=(\bar{x},\bar{u})$ and $\delta=0$ yields the equivalence between the Aubin property of $S_{\operatorname{KKT}}$ at the origin for $(\bar{x},\bar{u})$ and
SOQC \eqref{SOQC} combined with the second-order kernel condition \eqref{aubin-property-equiv-cond}. 

Now
we prove the implication (ii)$\Rightarrow$(iv). Combining the second-order subdifferential chain rule taken from \cite[Corollary~3.76]{Mordukhovich2006} with the second-order subdifferential sum rule from \cite[Proposition~1.121] {Mordukhovich2006} gives us
\begin{eqnarray*}
    D^{*}(\partial\varphi)(\bar{x},0)(d)\subset\bigcup_{u\in\Lambda(\bar{x},0,0)}[\nabla_{xx}^{2}L(\bar{x},u)d+\nabla F(\bar{x})^{T}D^{*}(\partial g)(F(\bar{x}),u)(\nabla F(\bar{x})d)].
\end{eqnarray*}
Take the polar operation of both sides of \eqref{SRCQ}  and observe by \cite[Proposition 3.3]{Tangwang2024} with 
\cite[Theorem 13.57]{Rockafellar1998} that SOQC \eqref{SOQC} implies SRCQ \eqref{SRCQ} and then $\Lambda(\bar{x},0,0)=\{\bar{u}\}$ is a singleton. Consequently, the statement (ii) together with Lemma \ref{lemma-aubin-property-Mordukhovich-criteria} tells us $\{d\;|\;0\in D^{*}(\partial\varphi)(\bar{x},0)(d)\}=\{0\}$ and  $\partial\varphi$ is metric regular at $\bar{x}$ for the origin. Let $\lambda>0$ be sufficiently small. By \cite[Theorem~4.4 and Proposition~4.6]{PoliquinRockafellar1996}, the envelope function $e_{\lambda\varphi}$ is of class $\mathcal{C}^{1}$ with a local minimizer $\bar{x}$ and the gradient mapping $\nabla e_{\lambda\varphi}$ is Lipschitz continuous satisfying
\begin{eqnarray*}
    \nabla e_{\lambda\varphi}(x)=\frac{x-\operatorname{Prox}_{\lambda\varphi}(x)}{\lambda}=[\lambda I+(\partial\varphi)^{-1}]^{-1}(x). 
\end{eqnarray*}
Consequently, the mapping $\nabla e_{\lambda\varphi}$ inherits the metric regularity at $\bar{x}$ for the origin from $\partial\varphi$ by  Lemma \ref{lemma-aubin-property-Mordukhovich-criteria}. 
Taking into account of \cite[Corollary~5.5]{Mohammadi2020} and \cite[Theorem~13.40 and Proposition~13.32]{Rockafellar1998}, $\varphi$ is twice-epidifferentiable around $(\bar{x},0)$ and both mappings  $\operatorname{Prox}_{\lambda\varphi}$ and $\nabla e_{\lambda\varphi}$ are directionally differentiable around $\bar{x}$. Applying \cite[Theorem~3.3]{Fusek2001}, \cite[Lemma~1.1]{Rabinowitz1975} and \cite[Theorem~1.2]{BarretoFenilleHartmann2016}, it follows that $(\nabla e_{\lambda\varphi})^{-1}=\lambda I+(\partial\varphi)^{-1}$ is single-valued around the origin. Combining this with the metric regularity of $\partial\varphi$, the mapping $(\partial\varphi)^{-1}$ is single-valued Lipschitz continuous around the orgin and therefore $\bar{x}$ is a tilt-stable local minimzer of $\varphi$. Implication (ii)$\Rightarrow$(iv) justified.

To proceed further, recall that SOQC \eqref{SOQC} implies SRCQ \eqref{SRCQ}. Then Lemma~\ref{lemma-rebust-SOQC} gives us neighborhoods $\widehat{\mathcal{U}}$ of $(\bar{x},\bar{u})$ and $\widehat{\mathcal{V}}$ of the origin such that SOQC holds at $x$ for $u\in\Lambda(x,\delta,0)$, and consequently $\Lambda(x,\delta,0)$ is a singleton, where $(x,u)\in\widehat{\mathcal{U}}$, $\delta\in\widehat{\mathcal{V}}$, and $(x,u)\in S_{\operatorname{KKT}}(\delta,0)$. Thus implication (iv)$\Rightarrow$(iii) follows from Lemma~\ref{lemma-tilt-stability-second-order-subdifferential-cond}.

Next we aim at verifying implication (iii)$\Rightarrow$(iv). Recall that if both SOQC \eqref{SOQC} and the second-order subdifferential condition \eqref{second-order-subdifferential-condition-single-dual} hold at $x$, then $\Lambda(x,a,0)$ is a singleton with $\|a\|$ sufficiently small. Employ \cite[Theorem 3.3]{Chieu2018} and  suppose on the contrary that there exist sequences $x^{k}\rightarrow \bar{x}$, $a^{k}\rightarrow 0$ as $k\rightarrow\infty$  with $a^{k}\in\partial \varphi(x^{k})$, $w^{k}\in D(\partial \varphi)(x^{k},a^{k})(d^{k})$ such that
 $\langle d^{k},w^{k}\rangle\leq 0$ for large $k$. Consider $d^{k}\neq 0$ for such $k$ and, by passing to a subsequence if necessary, assume that $\|d^{k}\|=1$ for large $k$ and that $d^{k}\rightarrow d$ with $\|d\|=1$. By \cite[Lemma~4.1]{MordukhovichTangWang-tilt-2025}, there exists $u^{k}\in\Lambda(x^{k},a^{k},0)$ with $u^{k}\rightarrow\bar{u}$ and $v^{k}\in D(\partial g)(F(x^{k}),u^{k})(\nabla F(x^{k})d^{k})$ such that
 \begin{equation*}
\begin{aligned}
\langle d^{k},w^{k}\rangle&=\langle\nabla_{xx}^{2}L(x^{k},u^{k})d^{k},d^{k}\rangle+d^{2}g(F(x^{k}),u^{k})(\nabla F(x^{k})d^{k})\\
&=\langle\nabla_{xx}^{2}L(x^{k},u^{k})d^{k},d^{k}\rangle+\langle \nabla F(x^{k})d^{k},v^{k}\rangle.
\end{aligned}
\end{equation*}
 By \cite[Theorem 13.57]{Rockafellar1998} and \cite[Proposition 3.2]{MordukhovichTangWang-tilt-2025}, taking a subsequence if necessary and letting $k\rightarrow\infty$, the coderivative robustness tells us that there exist $v\in D^{*}(\partial g)(F(\bar{x}),\bar{u})(\nabla F(\bar{x})d)$ with
 \begin{equation*}
 \langle d^{k},w^{k}\rangle\rightarrow\langle\nabla_{xx}^{2}L(\bar{x},\bar{u})d,d\rangle+\langle \nabla F(\bar{x})d,v\rangle\leq0,
 \end{equation*}
 which contradicts the second-order subdifferential condition \eqref{second-order-subdifferential-condition-single-dual} and thus verifies (iv). 

Now we show the implication (iii)$\Rightarrow$(ii). If both SOQC \eqref{SOQC} and the second-order subdifferential condition \eqref{second-order-subdifferential-condition-single-dual} hold at $\bar{x}$, then $\Lambda(\bar{x},0,0)$ is a singleton. Suppose that there exists $d\neq0$ with $0\in\nabla^{2}_{xx}L(\bar{x},\bar{u})d+\nabla F(\bar{x})^{T}D^{*}(\partial g)(F(\bar{x}),\bar{u})(\nabla F(\bar{x})d)$. This gives us a vector $v\in D^{*}(\partial g)(F(\bar{x}),\bar{u})(\nabla F(\bar{x})d)$ with $-\nabla_{xx}^{2}L(\bar{x},\bar{u})d=\nabla F(\bar{x})^{T}v$. Therefore,
\begin{equation*}
\begin{aligned}
\min_{z\in D^{*}(\partial g)(F(\bar{x}),\bar{u})(\nabla F(\bar{x})d)}\left\langle \nabla F(\bar{x})d,z\right\rangle
&\leq\left\langle\nabla F(\bar{x})d,v\right\rangle=-\left\langle\nabla_{xx}^{2}L(\bar{x},\bar{u})d,d\right\rangle,
\end{aligned}
\end{equation*}
which tells us in turn that the chosen vector $d\ne 0$ satisfies  the inequality
\begin{equation*}
\left\langle\nabla^{2}_{xx}L(\bar{x},\bar{u})d,d\right\rangle+\min_{z\in D^{*}(\partial g)(F(\bar{x}),\bar{u})(\nabla F(\bar{x})d)}\left\langle \nabla F(\bar{x})d,z\right\rangle\leq 0.
\end{equation*}
This contradicts the second-order subdifferential  condition \eqref{second-order-subdifferential-condition-single-dual} and hence yields (iii)$\Rightarrow$(ii)$\Leftrightarrow$(i). 
	 
Further, it follows from Lemma~\ref{lemma:equivalent-second-order-QC}, (iv), Proposition~\ref{prop-parabolically-regular-delta-epi-f}, and \cite[Theorem~4.1]{MordukhovichTangWang-tilt-2025} that (i) is equivalent to the Aubin property of the KKT mapping for the auxiliary problem \eqref{eq:problem-equivalent} at the origin for $((\bar{x},g(F(\bar{x}))),(\bar{u},-1))$. Strong regularity of $((\bar{x},g(F(\bar{x}))),(\bar{u},-1))$ for the KKT system associated with problem \eqref{eq:problem-equivalent} directly implies the fulfillment of this property of $(\bar{x},\bar{u})$ for the KKT system \eqref{eq:kkt-perturbed-cop}, which consequently guarantees the Aubin property of the solution mapping. Moreover, if $(\bar{x}, g(F(\bar{x})))$ is a fully stable minimizer of problem \eqref{eq:problem-equivalent}, then $\bar{x}$ is fully stable for problem \eqref{comp-prob}, since the full stability yields the tilt-stability of $\bar{x}$. Remembering that $\bar{x}$ is a local minimizer of $\varphi$ allows us to deduce from \cite[Lemma~4.2]{Tangwang2024}, \eqref{inverse-graphical-derivative}, and \eqref{sum-rule-graphical-derivative} that 
$$
\begin{array}{ll}
&\mathcal{C}(F(\bar{x}),g(F(\bar{x})))=\{(d,\beta)\;|\;dg(F(\bar{x}))(d)=\langle\bar{u},d\rangle,\;\beta=0\},\\
&\operatorname{dom}D^{*}\mathcal{N}_{\operatorname{epi}g}((F(\bar{x}),g(F(\bar{x}))),(\bar{u},-1))\subset\{(d,0)\;|\;d\in\operatorname{aff}\{y\;|\;dg(F(\bar{x}))(y)=\langle\bar{u},y\rangle\}\}.
\end{array}
$$
Then Lemmas~\ref{prop:coderivative-equiv-partial-f-partial-normal-epif} and \ref{lemma:equivalent-second-order-QC}  together with 
(i) and (ii) ensure that the Aubin property of the KKT system for problem \eqref{canonical-perturbed-constrained-prob} also holds at the origin for $((F(\bar
x),g(F(\bar
x))),(\bar{u},-1))$. By Lemma \ref{lemma:equivalent-second-order-QC} and \cite[Theorem~1]{KlatteKummer2013} it follows that the nondegeneracy condition \eqref{nondegeneracy-cond} holds at $\bar{x}$ for $\bar{u}$. The equivalences in (iii)$\Leftrightarrow$(v)$\Leftrightarrow$(vi) are due to Theorem~\ref{theorem-SOQC-nondegenerate-cond} and \cite[Theorem~5.6]{MordukhovichNghiaRockafellar2015}. It follows from \cite[Corollary~1.6]{BenkoRockafellar2024} that (vii)$\Rightarrow$(viii). Implication (v)$\Rightarrow$(viii) can be derived by mimicking the proof of \cite[Corollary~2.2]{Robinson1980}. Combining finally Lemmas~\ref{lemma-aubin-property-Mordukhovich-criteria} and \ref{lemma-strong-regularity-strict-graphical-derivative}, (viii) implies that the Aubin property of $S_{\operatorname{KKT}}$ holds at the origin for $(\bar{x},\bar{u})$ and there exist neighborhoods $\mathcal{U}$ of $(\bar{x},\bar{u})$ and $\mathcal{V}$ of the origin such that for any $\delta\in\mathcal{V}$, the set $S_{\operatorname{KKT}}(\delta,0)\cap\mathcal{U}$ is a singleton. We arrive at (viii)$\Rightarrow$(i) and  (viii)$\Rightarrow$(iv). Therefore (viii) tells us $\varphi$ is variationally strongly convex at $\bar{x}$ for the origin by \cite[Proposition 2.9]{KhanhMordukhovichPhat2023} and then (viii)$\Rightarrow$(vii) follows by \cite[Theorem~4.2]{BenkoRockafellar2024}. This readily completes the proof.
\end{proof}\vspace*{-0.05in}

To proceed further, consider the mapping
\begin{align}\label{Z}
Z(x,u):=\left[\begin{array}{c}
\nabla_{x}L(x,u)\\
u-\operatorname{Prox}_{g^{*}}(F(x)+u)
\end{array}\right]
\end{align}
and present the following lemma, which allows us to establish the second set of equivalences between well-posedness properties in composite optimization.\vspace*{-0.05in}

\begin{lemma}\label{lemma-SOQC-SSOSC-nonsigular-partial-Z}
Let $(\bar{x},\bar{u})$ be a solution to the KKT system \eqref{eq:kkt-perturbed-cop} with $(a,b)=(0,0)$, where $g$ satisfies {\rm(A1)} and {\rm(A2)} at $F(\bar{x})$ for $\bar{u}$. If SOQC \eqref{SOQC} and SSOSC \eqref{SSOSC} hold at $\bar{x}$, then all the matrices from the generalized Jacobian $\mathcal{J} Z(\bar{x},\bar{u})$ are nonsingular.
\end{lemma}
\begin{proof}
SOQC \eqref{SOQC} implies that $\Lambda(\bar{x},0,0)=\{\bar{u}\}$ is a singleton. For any $E\in\mathcal{J} Z(\bar{x},\bar{u})$, $\Delta x\in\mathcal{R}^{n}$, and $\Delta u\in\mathcal{R}^{m}$ with $E(\Delta x,\Delta u)=0$, we find $U\in\mathcal{J}\operatorname{Prox}_{g}(F(\bar{x})+\bar{u})$ such that
\begin{align}\label{E-dx-du-eq-1}
&\nabla_{xx}^{2}L(\bar{x},\bar{u})\Delta x+\nabla F(\bar{x})^{T}\Delta u=0,\\
&\nabla F(\bar{x})\Delta x-U(\nabla F(\bar{x})\Delta x+\Delta u)=0.\label{E-dx-du-eq-2}
\end{align}
It follows from the definition of $\Gamma_{g}(F(\bar{x}),\bar{u})$ in \eqref{Gamma} and equations \eqref{E-dx-du-eq-1}, \eqref{E-dx-du-eq-2} that
\begin{align*}
\Gamma_{g}(F(\bar{x}),\bar{u})(\nabla F(\bar{x})\Delta x)\leq\left\langle\nabla F(\bar{x})\Delta x,\Delta u\right\rangle=-\left\langle\nabla_{xx}^{2}L(\bar{x},\bar{u})\Delta x,\Delta x\right\rangle.
\end{align*}
Combining this with SSOSC \eqref{SSOSC} yields $\Delta x=0$. Substituting $\Delta x=0$ into \eqref{E-dx-du-eq-1} and \eqref{E-dx-du-eq-2}, we get 
$\Delta u\in\operatorname{ker}(\nabla F(\bar{x})^{T})\cap\operatorname{ker}U$. Applying \cite[Corollary~4.1]{Tangwang2024}, \eqref{inverse-coderivative}, and \eqref{sum-rule-coderivative} leads us to
\begin{align*}
\operatorname{ker}U\subset\left\{d\,\left|\,0\in D^{*}\operatorname{Prox}_{g}(F(\bar{x})+\bar{u})(d)\right.\right\}=-D^{*}(\partial g)(F(\bar{x}),\bar{u})(0)
\end{align*}
and thus ensures that SOQC \eqref{SOQC} implies $\Delta u=0$, which completes the proof.
\end{proof}

Now we are ready to establish the second set of equivalent relationships between well-posedness properties with their pointbased second-order characterizations.\vspace*{-0.05in}

\begin{theorem}\label{theorem-aubin-equivalence-conditions}
Let $\bar{x}$ be a local minimizer of problem \eqref{comp-prob} with $\bar{u}\in\Lambda(\bar{x},0,0)$ as a corresponding Lagrange multiplier,  and let $g$ satisfy both assumptions {\rm(A1), (A2)} at $F(\bar{x})$ for $\bar{u}$. Then we have the following equivalent assertions: 
	
{\bf(i)} The mapping $S_{\operatorname{KKT}}$ has the Aubin property at the origin for $(\bar{x},\bar{u})$. 
	
{\bf(ii)} SOQC \eqref{SOQC} holds together with SSOSC  at $(\bar{x},\bar u)$ given by
\begin{align}\label{SSOSC-SVF}
\left\langle\nabla_{xx}^{2}L(\bar{x},\bar{u})d,d\right\rangle+\Gamma_{g}(F(\bar{x}),\bar{u})(\nabla F(\bar{x})d)>0\;\mbox{ for all }\;d\ne 0.
\end{align} 

{\bf(iii)} SOQC \eqref{SOQC} is satisfied at the point $\bar{x}$, which is a tilt-stable local minimizer of the function $\varphi$ from \eqref{comp-prob}. 

{\bf(iv)} Each matrix from the generalized Jacobian $\mathcal{J}Z(\bar{x},\bar{u})$ of $Z$ in \eqref{Z} is nonsingular.
	
{\bf(v)} The pair $(\bar{x},\bar{u})$ is a strongly regular point of the KKT system \eqref{eq:kkt-perturbed-cop}.
	
{\bf(vi)} The KKT mapping $\Phi$ from \eqref{Phi} is strongly metrically regular at $(\bar{x},\bar{u})$ for $(0,0)$.
	
{\bf(vii)}  The solution $(\bar{x},\bar{u})$ to the KKT system \eqref{eq:kkt-perturbed-cop} is fully stable with respect to the origin. 
\end{theorem}\vspace*{-0.15in}
\begin{proof}
Employing Theorem~\ref{theorem-aubin-property-SOQC-SSOSC}, Lemma~\ref{lemma-SOQC-SSOSC-nonsigular-partial-Z}, and \cite[Proposition~3.3]{MordukhovichTangWang-tilt-2025} verifies the relationships in
(i)$\Leftrightarrow$(ii)$\Leftrightarrow$(iii)$\Rightarrow$(iv). It follows from Clarke’s inverse function theorem \cite{Clarke1976} combined with \cite[Lemma~5.1]{Tangwang2024} that (iv)$\Rightarrow$(v). Implication (vi)$\Rightarrow$(i) follows from Lemmas \ref{lemma-aubin-property-Mordukhovich-criteria} and \ref{lemma-strong-regularity-strict-graphical-derivative} and  (v)$\Rightarrow$(vi) follows from \cite[Corollary~2.2]{Robinson1980} similarly. Repeating the proof of (viii)$\Rightarrow$(vii) in Theorem \ref{theorem-aubin-equivalence-conditions} verifies the fulfillment of implication (vi)$\Rightarrow$(vii). From \cite[Corollary~1.6]{BenkoRockafellar2024}, we finally deduce that (vii)$\Rightarrow$(vi) and  thus complete the proof of the theorem.
\end{proof}\vspace*{-0.05in}

Summarizing the results of Theorems~\ref{theorem-aubin-property-SOQC-SSOSC} and \ref{theorem-aubin-equivalence-conditions}, we finally arrive at the equivalences between the major well-posedness properties of KKT systems associated with general ${\cal C}^2$-cone reducible problems of composite optimization and their complete second-order characterizations.\vspace*{-0.05in}

\begin{corollary}
$\bar{x}$ be a local minimizer of problem \eqref{comp-prob} with $\bar{u}\in\Lambda(\bar{x},0,0)$ as a corresponding Lagrange multiplier, and let $g$ be $\mathcal{C}^{2}$-cone reducible while satisfying {\rm(A2)} at $F(\bar{x})$ for $\bar{u}$. Then the following assertions are equivalent:
	
{\bf(i)} The KKT solution mapping $S_{\operatorname{KKT}}$ has the Aubin property at the origin for $(\bar{x},\bar{u})$. 
	
{\bf(ii)} SOQC \eqref{SOQC} holds together with SSOSC \eqref{SSOSC-SVF} at $(\bar x,\bar u)$.
	
{\bf(iii)} SOQC \eqref{SOQC} is satisfied at the point $\bar{x}$, which is a tilt-stable local minimizer of the function $\varphi$ from \eqref{comp-prob}.

{\bf(iv)} Each matrix from the generalized Jacobian $\mathcal{J}Z(\bar{x},\bar{u})$ of $Z$ in \eqref{Z} is nonsingular.

{\bf(v)} The pair $(\bar{x},\bar{u})$ is a strongly regular point of the KKT system \eqref{eq:kkt-perturbed-cop}.

{\bf(vi)} The constraint nondegeneracy condition \eqref{nondegeneracy-cond} holds at $\bar{x}$ for $\bar{u}$ and $\bar{x}$ is a fully stable local minimizer of problem \eqref{comp-prob-perturb}.

{\bf(vii)} The solution $(\bar{x},\bar{u})$ to the KKT system \eqref{eq:kkt-perturbed-cop} is fully stable with respect to the origin.
	
{\bf(viii)} The KKT mapping $\Phi$ from \eqref{Phi} is strongly metrically regular at $(\bar{x},\bar{u})$ for $(0,0)$.
\end{corollary}\vspace*{-0.3in}

\section{Concluding Remarks and Future Research}\label{sec:Conclusion}\vspace*{-0.05in}

This paper establishes {\em equivalent relationships} between various {\em well-posedness} properties of general composite optimization problems and associated KKT systems with providing {\em second-order characterizations} of these properties via pointbased constructions of {\em second-order generalized differentiation}. In particular, we establish close relationships between the {\em Aubin property} and {\em strong regularity} of solution mappings to KKT systems, as well as the {\em tilt stability} of local minimizers, in the {\em ${\cal C}^2$-cone reducible} setting of composite optimization with {\em parabolically regular} data. It is also shown, among other results, that these properties and their equivalents are characterized by the simultaneous fulfillment of certain {\em constraint nondegeneracy} and {\em robust second-order subdifferential conditions} expressed {\em precisely} at the reference solution and satisfying {\em comprehensive calculus} rules. The obtained results are of a fundamental nature, while being instrumental for {\em stability} and {\em sensitivity analysis} of optimization-related problems, as well as for the design and justification of {\em numerical algorithms} in constrained optimization.\vspace*{0.03in}

Our future research will be conducted in various directions. Among them, we plan to reveal---besides the broad settings discussed above---other classes of composite optimization problems satisfying assumption (A2) with {\em explicit computations} of the second-order variational function $\Gamma_f$. 
Recalling that strong regularity, Lipschitzian stability of KKT systems, and tilt stability of local minimizers play a crucial role in justifying numerical algorithms of constrained optimization, our attention will be paid to the implementation of these ideas and the obtained results to {\em new classes} of composite optimization problems. Moreover, we intend to investigate some of the well-posedness properties studied in this paper {\em without imposing nondegeneracy} conditions.\vspace*{-0.2in} 
\section*{Acknowledgments} We would like to thank Associate Editor, and anonymous referees for their valuable suggestions and insightful comments, which allowed us to greatly improve the quality of the manuscript. The authors are also gratefully indebted to Helmut Gfrerer for observing serious inaccuracies in Section~5 and very useful suggestions that allowed us to significantly improve the original presentation.\vspace*{-0.2in} 

\section*{Disclosure of interest}
Research of Boris Mordukhovich was partly supported by the US National Science Foundation under grant DMS-2204519 and by the Australian Research Council under Discovery Project DP250101112. Research of Peipei Tang was supported partly by the National Natural Science Foundation of China under Grant 12671637 and by Zhejiang Provincial Natural Science Foundation of China
under Grant No. LMS26A010022. Research of Chengjing Wang was partly supported by the National Natural Science Foundation of China under Grant U21A20169, by the Zhejiang Provincial Natural Science Foundation of China under Grant LTGY23H240002.
The authors report there are no competing interests to declare.\vspace*{-0.2in}
\section*{Data Availability}
Data sharing is not applicable to this paper as no new data were created or analyzed.\vspace*{-0.2in}

\bibliographystyle{amsplain}
\bibliography{ref.bib}

\end{document}